\documentclass[12pt]{amsart}
\usepackage{}

\usepackage{amsmath}
\usepackage{amsfonts}
\usepackage{amssymb}
\usepackage[all]{xy}           

\usepackage{bbding}
\usepackage{txfonts}
\usepackage{amscd}

\usepackage[shortlabels]{enumitem}
\usepackage{ifpdf}
\ifpdf
\usepackage[colorlinks,final,
hyperindex]{hyperref}
\else
\usepackage[colorlinks,final,
hyperindex]{hyperref}
\fi
\usepackage{tikz}
\usepackage[active]{srcltx}

\makeatletter

\numberwithin{equation}{section}

\newtheorem{thm}{Theorem}[section]
\newtheorem{lem}[thm]{Lemma}
\newtheorem{cor}[thm]{Corollary}
\newtheorem{pro}[thm]{Proposition}
\newtheorem{ex}[thm]{Example}

\newtheorem{defi}[thm]{Definition}

\newcommand {\emptycomment}[1]{}

\newcommand{\g}{\mathfrak g}
\newcommand{\gl}{\mathfrak {gl}}
\newcommand{\kl}{\mathfrak l}
\newcommand{\kr}{\mathfrak r}

\newcommand{\fu}{\mathbf u}
\newcommand{\fl}{\mathbf l}
\newcommand{\fr}{\mathbf r}

\newcommand{\ad}{\mathrm{ad}}
\newcommand{\Img}{\mathrm{Im}}

\newcommand{\q}{\mathrm{q}}

\def\id{\mathop {\fam0 id}\nolimits}

\begin{document}
\title[Quasi-triangular Novikov bialgebras and related bialgebra structures]
{Quasi-triangular Novikov bialgebras and related bialgebra structures}

\author{Zhanpeng Cui}
\address{School of Mathematics and Statistics, Henan University, Kaifeng 475004, China}
\email{czp15824833068@163.com}

\author{Bo Hou}
\address{School of Mathematics and Statistics, Henan University, Kaifeng 475004, China}
\email{bohou1981@163.com}
		
\vspace{-5mm}
	
\begin{abstract}
We introduce the notion of quasi-triangular Novikov bialgebras, which constructed from
solutions of the Novikov Yang-Baxter equation whose symmetric parts are invariant.
Triangular Novikov bialgebras and factorizable Novikov bialgebras are important
subclasses of quasi-triangular Novikov bialgebras. A factorizable Novikov bialgebra
induces a factorization of the underlying Novikov algebra and the double
of any Novikov bialgebra naturally admits a factorizable Novikov bialgebra structure.
Moreover, we introduce the notion of quadratic Rota-Baxter Novikov algebras and show that
there is an one-to-one correspondence between factorizable Novikov bialgebras and
quadratic Rota-Baxter Novikov algebras of nonzero weights.
Finally, we obtain that the Lie bialgebra induced by a Novikov bialgebra and a quadratic
right Novikov algebra is quasi-triangular (resp. triangular, factorizable) if the
Novikov bialgebra is quasi-triangular (resp. triangular, factorizable), and under
certain conditions, the Novikov bialgebra induced by a differential infinitesimal bialgebra
is quasi-triangular (resp. triangular, factorizable) if the differential infinitesimal
bialgebra is quasi-triangular (resp. triangular, factorizable).
\end{abstract}
	
\keywords{quasi-triangular Novikov bialgebra, factorizable Novikov bialgebras,
Novikov Yang-Baxter equation, quadratic Rota-Baxter Novikov algebra,
quasi-triangular differential infinitesimal bialgebra,
quasi-triangular Lie bialgebra.}
\subjclass[2010]{16T10, 16T25, 17B62, 17A60.}
	
\maketitle
\vspace{-4mm}
\tableofcontents

\section{Introduction}\label{sec:intr}
Novikov algebras were independently introduced in the study of Hamiltonian operators in
the formal calculus of variations by Gelfand and Dorfman \cite{GD1,GD2} as well
as in the connection with linear Poisson brackets of hydrodynamic type by Balinskii
and Novikov \cite{BN}. It was shown in \cite{Xu} that Novikov algebras correspond to
a class of Lie conformal algebras which give an axiomatic description of the singular
part of operator product expansion of chiral fields in conformal field theory \cite{K1}.
Novikov algebras are a special class of pre-Lie algebras (also called left-symmetric algebras),
which are tightly connected with many fields in mathematics and physics such as affine
manifolds and affine structures on Lie groups \cite{Ko}, convex homogeneous cones
\cite{Vin}, vertex algebras \cite{BLP, BK} and so on. Recently, multi-Novikov algebras
arose from regularity structure of stochastic PDEs \cite{BHZ,BD}.

A bialgebra structure on a given algebraic structure is obtained as a coalgebra
structure together which gives the same algebraic structure on the dual space with
a set of compatibility conditions between the multiplications and comultiplications.
One of the most famous examples of bialgebras is the Lie bialgebra \cite{Dri}, and
more importantly, there have been a lot of bialgebra theories for other algebra
structures that essentially follow the approach of Lie bialgebras such as
antisymmetric infinitesimal bialgebras \cite{Agu,Bai}, left-symmetric bialgebras
\cite{Bai1}, Jordan bialgebras \cite{Zhe}, and mock-Lie bialgebras \cite{BCHM}.
These bialgebras have a common property, that is, they have equivalent characterizations
in terms of Manin triples which correspond to nondegenerate bilinear forms on the algebra
structures satisfying certain conditions. In addition, it is noteworthy that one can study
the bialgebra theories in the sense that they give a distributive law in \cite{LMW}.

A recent study \cite{HBG} introduced the notion of a Novikov bialgebra to construct
infinite-dimensional Lie bialgebras via the Novikov bialgebra affinization.
The concept of a Novikov Yang-Baxter equation was proposed to construct
Novikov bialgebras. In this paper, we further study some special Novikov bialgebras
such as quasi-triangular Novikov bialgebras, triangular Novikov bialgebras and
factorizable Novikov bialgebras. Factorizable bialgebra is an important subclass of
quasi-triangular bialgebra, which give rise to a natural factorization of the
underlying algebras. Recently, the factorizable theory has been studied for
Lie bialgebras \cite{LS}, antisymmetric infinitesimal bialgebras \cite{SW},
pre-Lie bialgebras \cite{WBLS}, Leibniz bialgebras \cite{BLST},
mock-Lie bialgebras \cite{CH} and so on. Here, we study the quasi-triangular
Novikov bialgebras and factorizable Novikov bialgebras, and give an one-to-one
correspondence between factorizable Novikov bialgebras and quadratic Rota-Baxter
Novikov algebras of nonzero weights. Moreover, we consider the relationship between
quasi-triangular Novikov bialgebras and solutions of Novikov Yang-Baxter equation,
as well as the relationship between quasi-triangular Lie bialgebras and quasi-triangular
Novikov bialgebras. We get the first main theorem:

\smallskip\noindent
{\bf Theorem I} (Theorem \ref{thm:indu-sLiebia})
{\it Let $(A, \diamond, \delta)$ be a Novikov bialgebra and $(B, \circ, \omega)$ be a quadratic
right Novikov algebra, and $(A\otimes B, [-,-], \tilde{\Delta})$ be the induced
Lie bialgebra by $(A, \diamond, \delta)$ and $(B, \circ, \omega)$. Then,
\begin{enumerate}\itemsep=0pt
\item $(A\otimes B, [-,-], \tilde{\Delta})$ is quasi-triangular if
     $(A, \diamond, \delta)$ is quasi-triangular;
\item $(A\otimes B, [-,-], \tilde{\Delta})$ is triangular if
     $(A, \diamond, \delta)$ is triangular;
\item $(A\otimes B, [-,-], \tilde{\Delta})$ is factorizable if
     $(A, \diamond, \delta)$ is factorizable.
\end{enumerate}}

Most recently, the relations between differential infinitesimal bialgebras and
Novikov bialgebras were established in \cite{HBG1}. It was shown that under some
additional conditions, every (commutative and cocommutative) differential infinitesimal
bialgebra derived a Novikov bialgebra. In this paper, we further study this construction
for quasi-triangular (resp. triangular, factorizable) Novikov bialgebras.

\smallskip\noindent
{\bf Theorem II} (Theorem \ref{thm:indu-spbia})
{\it Let $(A, \cdot, \partial, \theta)$ be an admissible differential algebra. Suppose
$(A, \cdot, \Delta, \partial, \theta)$ is a differential infinitesimal bialgebra
and $(A, \diamond_{\q}, \delta_{\q})$ is the Novikov bialgebra induced by
$(A, \cdot, \Delta, \partial, \theta)$. Then, if $\q=-\frac{1}{2}$, or $\theta$ is a
derivation on $(A, \cdot)$, we have
\begin{enumerate}\itemsep=0pt
\item $(A, \diamond_{\q}, \delta_{\q})$ is quasi-triangular if
     $(A, \cdot, \Delta, \partial, \theta)$ is quasi-triangular;
\item $(A, \diamond_{\q}, \delta_{\q})$ is triangular if
     $(A, \cdot, \Delta, \partial, \theta)$ is triangular;
\item $(A, \diamond_{\q}, \delta_{\q})$ is factorizable if
     $(A, \cdot, \Delta, \partial, \theta)$ is factorizable.
\end{enumerate}}

This paper is organized as follows. In Section \ref{sec:quasi-tri}, we show that a
solution of the Novikov Yang-Baxter equation in a Novikov algebra whose symmetric
part is invariant gives rise to a quasi-triangular Novikov bialgebra.
In Section \ref{sec:fact}, we introduce the notions of factorizable Novikov
bialgebras and show that the double of a Novikov bialgebra naturally has a
factorizable Novikov bialgebra structure. We also introduce the notion of
quadratic Rota-Baxter Novikov algebras and show that there is an one-to-one
correspondence between factorizable Novikov bialgebras and quadratic Rota-Baxter
Novikov algebras of nonzero weights. In Section \ref{sec:factdiff},
for some special admissible differential algebras, we show that a solution $r$ of
the NYBE in this admissible differential algebra such that $\mathfrak{s}(r)$ is
$\fu_{A}$-invariant is also a solution of the NYBE in the induced Novikov algebra
such that the symmetric part is invariant. Therefore, under certain conditions, we
show that the quasi-triangular Novikov bialgebra induced by a quasi-triangular
(resp. triangular, factorizable) differential infinitesimal bialgebra is also
quasi-triangular (resp. triangular, factorizable). In \cite{HBG}, Hong, Bai and Guo
provide a method for constructing Lie bialgebras using Novikov bialgebras and
quadratic right Novikov algebras. In Section \ref{sec:factlie}, we consider the
Lie bialgebras constructed from quasi-triangular Novikov bialgebras by using this method,
and show that the Lie bialgebra induced by a Novikov bialgebra and a quadratic
right Novikov algebra is quasi-triangular (resp. triangular, factorizable) if the
Novikov bialgebra is quasi-triangular (resp. triangular, factorizable).

Throughout this paper, we fix $\Bbbk$ a field and characteristic zero.
All the vector spaces and algebras are over $\Bbbk$, and all tensor products are also
taking over $\Bbbk$. For any vector space $V$, we denote $V^{\ast}$ the dual space of $V$.
We denote the identity map by $\id$.

\bigskip
\section{Quasi-triangular Novikov bialgebras}\label{sec:quasi-tri}
In this section, we show that a solution of the Novikov Yang-Baxter equation
in a Novikov algebra whose symmetric part is invariant gives rise to a
quasi-triangular Novikov bialgebra.
We first recall some known facts about Novikov algebras and Novikov bialgebras.
Recall that a {\it pre-Lie algebra} $(A, \diamond)$ is a vector space with a binary
operation $\diamond$ satisfying the following condition:
$$
(a_{1}\diamond a_{2})\diamond a_{3}-a_{1}\diamond (a_{2}\diamond a_{3})
=(a_{2}\diamond a_{1})\diamond a_{3}-a_{2}\diamond (a_{1}\diamond a_{3}),
$$
for any $a_{1}, a_{2}, a_{3}\in A$. A {\it (left) Novikov algebra} $(A, \diamond)$ is a
pre-Lie algebra satisfying the following condition:
$$
(a_{1}\diamond a_{2})\diamond a_{3}=(a_{1}\diamond a_{3})\diamond a_{2},
$$
for any $a_{1}, a_{2}, a_{3}\in A$. Obviously, a commutative associative algebra
is a Novikov algebra. Let $(A, \cdot)$ be a commutative associative algebra,
and $\partial: A\rightarrow A$ be a derivation on $(A, \cdot)$, i.e.,
$\partial(a_{1}\cdot a_{2})=\partial(a_{1})\cdot a_{2}+a_{1}\cdot\partial(a_{2})$
for any $a_{1}, a_{2}\in A$. Then, $(A, \cdot, \partial)$ is said to be a
differential algebra. The classical example of a Novikov algebra was given by
S. Gelfand \cite{GD1} from differential algebras. Given a differential algebra
$(A, \cdot, \partial)$, the binary operation
$$
a_{1}\diamond_{\partial}a_{2}:=a_{1}\cdot\partial(a_{2}),
$$
for any $a_{1}, a_{2}\in A$, defines a Novikov algebra $(A, \diamond_{\partial})$,
which is called the Novikov algebra induced by a differential algebra $(A, \cdot, \partial)$.
It was shown that each Novikov algebra can be embedded into a differential
algebra \cite{BCZ}.

Let $(A, \diamond)$ and $(B, \diamond')$ be two Novikov algebras. A linear map $f:
A\rightarrow B$ is called a {\it Novikov algebra homomorphism} if for any $a_{1}, a_{2}\in A$,
$f(a_{1}\diamond a_{2})=f(a_{1})\diamond'f(a_{2})$. We define a product $\ast$ on
direct sum $A\oplus B$ of vector spaces by
$$
(a_{1}, b_{1})\ast(a_{2}, b_{2})=(a_{1}\diamond a_{2},\ \ b_{1}\diamond' b_{2}),
$$
for any $a_{1}, a_{2}\in A$ and $b_{1}, b_{2}\in B$. Then $(A\oplus B, \ast)$ is also
a Novikov algebra, which is called the {\it direct sum Novikov algebra}
of $(A, \diamond)$ and $(B, \diamond')$.

\begin{defi}\label{def:rep}
A {\rm representation} of a Novikov algebra $(A, \diamond)$ is a triple
$(V, \kl, \kr)$, where $V$ is a vector space and $\kl, \kr:
A\rightarrow\gl(V)$ are linear maps satisfying
\begin{align*}
&\kl(a_{1}\diamond a_{2}-a_{2}\diamond a_{1})=\kl(a_{1})\kl(a_{2})-\kl(a_{2})\kl(a_{1}),
\qquad\qquad\quad\;\, \kl(a_{1}\diamond a_{2})=\kr(a_{2})\kl(a_{1}),\\
&\kl(a_{1})\kr(a_{2})-\kr(a_{2})\kl(a_{1})=\kr(a_{1}\diamond a_{2})-\kr(a_{2})\kr(a_{1}),
\qquad\qquad\kr(a_{1})\kr(a_{2})=\kr(a_{2})\kr(a_{1}),
\end{align*}
for any $a_{1}, a_{2}\in A$.
\end{defi}

Let $(A, \diamond)$ be a Novikov algebra. Define linear maps
$\fl_{A}, \fr_{A}: A\rightarrow\gl(A)$ by $\fl_{A}(a_{1})(a_{2})=a_{1}\diamond a_{2}$
and $\fr_{A}(a_{1})(a_{2})=a_{2}\diamond a_{1}$ for any $a_{1}, a_{2}\in A$.
Then, $(A, \fl_{A}, \fr_{A})$ is a representation of $(A, \diamond)$, which is called
the {\it adjoint representation} of $(A, \diamond)$.
Let $V$ be a vector space. Denote the standard pairing between the dual space
$V^{\ast}$ and $V$ by
$$
\langle-,-\rangle:\quad V^{\ast}\otimes V\rightarrow \Bbbk, \qquad\quad
\langle \xi,\; v \rangle:=\xi(v),
$$
for any $\xi\in V^{\ast}$ and $v\in V$. Let $V$, $W$ be two vector spaces. For a linear
map $\varphi: V\rightarrow W$, the transpose map $\varphi^{\ast}: W^{\ast}\rightarrow
V^{\ast}$ is defined by
$$
\langle \varphi^{\ast}(\xi),\; v \rangle:=\langle\xi,\; \varphi(v)\rangle,
$$
for any $v\in V$ and $\xi\in W^{\ast}$. Let $(A, \cdot)$ be a Novikov algebra
and $V$ be a vector space. For a linear map $\psi: A\rightarrow\gl(V)$, the linear map
$\psi^{\ast}: A\rightarrow\gl(V^{\ast})$ is defined by
$$
\langle\psi^{\ast}(a)(\xi),\; v\rangle:=-\langle\xi,\; \psi(a)(v)\rangle,
$$
for any $a\in A$, $v\in V$, $\xi\in V^{\ast}$. That is, $\psi^{\ast}(a)=\psi(a)^{\ast}$
for all $a\in A$.

It is well-known that $(V^{\ast}, -\mu^{\ast})$ is a representation of the
commutative associative algebra $(A, \cdot)$ if $(V, \mu)$ is a representation
of $(A, \cdot)$, and $(V^{\ast}, \kl^{\ast}+\kr^{\ast}, -\kr^{\ast})$ is a representation
of Novikov algebra $(A, \diamond)$ if $(V, \kl, \kr)$ is a representation of $(A, \diamond)$.
By direct calculation, we have

\begin{pro}\label{pro:dual}
Let $(A, \diamond)$ be a Novikov algebra and $(V, \kl, \kr)$ be a representation
of it. Then $(V^{\ast}, \kl^{\ast}+\kr^{\ast}, -\kr^{\ast})$ is also a
representation of $(A, \diamond)$. We call it the {\rm dual representation of
$(V, \kl, \kr)$}.
\end{pro}

In particular, for any Novikov algebra $(A, \diamond)$, $(A^{\ast},
\fl_{A}^{\ast}+\fr_{A}^{\ast}, -\fr_{A}^{\ast})$ is a representation of
$(A, \diamond)$, which is the dual representation of the adjoint representation
$(A, \fl_{A}, \fr_{A})$, called the {\it coadjoint representation} of $(A, \diamond)$.
In \cite{HBG}, Hong, Bai and Guo have studied the bialgebra theory of Novikov algebras.
We now recall the matched pairs, Manin triples and bialgebras of Novikov algebras.
Let $(A, \diamond)$ and $(B, \diamond')$ be two Novikov algebras, $(B, \kl_{A}, \kr_{A})$
and $(A, \kl_{B}, \kr_{B})$ be representations of $(A, \diamond)$ and $(B, \diamond')$
respectively. If $A\oplus B$ under the product
\begin{align}
(a_{1}, b_{1})\ast(a_{2}, b_{2})
=\big(a_{1}\diamond a_{2}+\kl_{B}(b_{1})(a_{2})+\kr_{B}(b_{2})(a_{1}),\ \
b_{1}\diamond'b_{2}+\kl_{A}(a_{1})(b_{2})+\kr_{A}(a_{2})(b_{1})\big),  \label{product}
\end{align}
for all $a_{1}, a_{2}\in A$ and $b_{1}, b_{2}\in B$, is a Novikov algebra, then
$(A, B, \kl_{A}, \kr_{A}, \kl_{B}, \kr_{B})$ is called a {\rm matched pair of
Novikov algebras} $(A, \diamond)$ and $(B, \diamond')$. The Novikov algebra
$(A\oplus B,\; \ast)$ is called a {\rm crossed product} of $(A, \diamond)$ and
$(B, \diamond')$, denoted by $A\bowtie B$.
Next, we give notions of some special bilinear form Novikov algebras.
Let $\mathfrak{B}(-, -)$ be a bilinear form on a Novikov algebra $(A, \diamond)$.
Recall that
\begin{itemize}
\item[-] $\mathfrak{B}(-, -)$ is called {\it nondegenerate} if $\mathfrak{B}(a_{1},
     a_{2})=0$ for any $a_{2}\in A$, then $a_{1}=0$;
\item[-] $\mathfrak{B}(-, -)$ is called {\it symmetric} if $\mathfrak{B}(a_{1}, a_{2})= \mathfrak{B}(a_{2}, a_{1})$, for any $a_{1}, a_{2}\in A$;
\item[-] $\mathfrak{B}(-, -)$ is called {\it invariant} if
$$
\mathfrak{B}(a_{1}\diamond a_{2},\; a_{3})+\mathfrak{B}(a_{2},\; a_{1}\diamond a_{3}
+a_{3}\diamond a_{1})=0,
$$
for any $a_{1}, a_{2}, a_{3}\in A$.
\end{itemize}
A {\it quadratic Novikov algebra} $(A, \diamond, \mathfrak{B})$ is a Novikov algebra
$(A, \diamond)$ equipped with a nondegenerate symmetric invariant bilinear form
$\mathfrak{B}(-,-)$. A {\it (standard) Manin triple of Novikov algebras} is a triple
of Novikov algebras $(\mathcal{A}=A\oplus A^{\ast},\; (A, \diamond),\; (A^{\ast}, \circ))$
for which
\begin{enumerate}
\item as a vector space, $\mathcal{A}$ is the direct sum of $A$ and $A^{\ast}$;
\item $(A, \diamond)$ and $(A^{\ast}, \circ)$ are Novikov subalgebras of $\mathcal{A}$;
\item the bilinear form on $\mathcal{A}=A\oplus A^{\ast}$ defined by
$$
\mathcal{B}\big((a_{1}, \xi_{1}),\;  (a_{2}, \xi_{2})\big)
:=\langle\xi_{1}, a_{2}\rangle+\langle\xi_{2}, a_{1}\rangle,
$$
for any $a_{1}, a_{2}\in A$ and $\xi_{1}, \xi_{2}\in A^{\ast}$, is invariant.
\end{enumerate}
A {\it Novikov coalgebra} $(A, \delta)$ is a vector space
$A$ with a linear map $\delta: A\rightarrow A\otimes A$ such that
\begin{align*}
(\tau\otimes\id)(\id\otimes\delta)\tau\delta
&=(\delta\otimes\id)\delta,\\
(\id\otimes\delta)\delta-(\tau\otimes\id)(\id\otimes\delta)\delta
&=(\delta\otimes\id)\delta-(\tau\otimes\id)(\delta\otimes\id)\delta.
\end{align*}
It is easy to see that, $(A, \delta)$ is a Novikov coalgebra if and only if
$(A^{\ast}, \diamond_{\delta}=\delta^{\ast})$ is a Novikov algebra. Recall \cite{HBG} that a
{\it Novikov bialgebra} is a tripe $(A, \diamond, \delta)$ where $(A, \diamond)$
is a Novikov algebra and $(A, \delta)$ is a Novikov coalgebra such that
\begin{align*}
&\qquad\qquad \delta(a_{1}\diamond a_{2})=(\fr_{A}(a_{2})\otimes\id)(\delta(a_{1}))
+(\id\otimes(\fl_{A}+\fr_{A})(a_{1}))(\delta(a_{2})+\tau(\delta(a_{2}))),\\
&\qquad\qquad\qquad(\fl_{A}+\fr_{A})(a_{1})\otimes\id)(\delta(a_{2}))
-(\id\otimes(\fl_{A}+\fr_{A})(a_{1}))(\tau(\delta(a_{2})))\\[-1mm]
&\qquad\qquad\quad=((\fl_{A}+\fr_{A})(a_{2})\otimes\id)(\delta(a_{1}))
-(\id\otimes(\fl_{A}+\fr_{A})(a_{2}))(\tau(\delta(a_{1}))),\\
&(\id\otimes\fr_{A}(a_{1})-\fr_{A}(a_{1})\otimes\id)(\delta(a_{2})+(\tau(\delta(a_{2}))))
=(\id\otimes\fr_{A}(a_{2})-\fr_{A}(a_{2})\otimes\id)(\delta(a_{1})+(\tau(\delta(a_{1})))),
\end{align*}
for all $a_{1}, a_{2}\in A$.

\begin{ex}\label{ex:2bialg}
Let vector space $A=\Bbbk\{e_{1}, e_{2}\}$. Define $e_{1}\diamond e_{1}=e_{2}$.
Then $(A, \diamond)$ is a 2-dimensional Novikov algebra. Define linear map $\delta: A
\rightarrow A\otimes A$ by $\delta(e_{1})=e_{2}\otimes e_{2}$ and $\delta(e_{2})=0$,
then we get a $2$-dimensional Novikov bialgebra $(A, \diamond, \delta)$.
\end{ex}

A Novikov bialgebra can be characterized by a Manin triple and a matched pair
of Novikov algebras.

\begin{thm}[\cite{HBG}]\label{thm:Novequi}
Let $(A, \diamond)$ be a Novikov algebra and $(A, \delta)$ a Novikov coalgebra.
Then the following conditions are equivalent.
\begin{enumerate}
\item There is a Manin triple $(A\oplus A^{\ast},\; (A, \diamond),\;
     (A^{\ast}, \diamond_{\delta}))$ of Novikov algebras;
\item $(A, A^{\ast}, \fl_{A}^{\ast}+\fr_{A}^{\ast}, -\fr_{A}^{\ast},
     \fl_{A^{\ast}}^{\ast}+\fr_{A^{\ast}}^{\ast}, -\fr_{A^{\ast}}^{\ast})$ is a matched
     pair of Novikov algebras;
\item $(A, \diamond, \delta)$ is a Novikov bialgebra.
\end{enumerate}
\end{thm}

Recall that a Novikov bialgebra $(A, \diamond, \delta)$
is called {\it coboundary} if there exists a $r\in A\otimes A$ such that
\begin{align}
\delta(a)=\delta_{r}(a)=-\big(\fl_{A}(a)\otimes\id
+\id\otimes(\fl_{A}+\fr_{A})(a)\big)(r),      \label{cobnov}
\end{align}
for any $a\in A$. In this case, we denote this coboundary Novikov bialgebra by
$(A, \diamond, \delta_{r})$.

\begin{pro}[\cite{HBG}]\label{pro:cob-nov}
Let $(A, \diamond)$ be a Novikov algebra, $r\in A\otimes A$, and $\delta_{r}:
A\rightarrow A\otimes A$ be a linear map defined by Eq. \eqref{cobnov}. Then $(A, \diamond,
\delta_{r})$ is a Novikov bialgebra if and only if for any $a, a_{1}, a_{2}\in A$,
\begin{align}
&\Big(\id\otimes(\fl_{A}(a_{2}\diamond a_{1})+\fl_{A}(a_{1})\fl_{A}(a_{2}))
+(\fl_{A}+\fr_{A})(a_{1})\otimes(\fl_{A}+\fr_{A})(a_{2})\Big)(r+\tau(r))=0, \label{bia1}\\
&\quad\Big((\fl_{A}+\fr_{A})(a_{1})\otimes(\fl_{A}+\fr_{A})(a_{2})
-(\fl_{A}+\fr_{A})(a_{2})\otimes(\fl_{A}+\fr_{A})(a_{1})\Big)(r+\tau(r))=0, \label{bia2}\\
&\Big((\fl_{A}+\fr_{A})(a_{1})\otimes\fr_{A}(a_{2})
-(\fl_{A}+\fr_{A})(a_{2})\otimes\fr_{A}(a_{1})
+\fr_{A}(a_{1})\otimes\fl_{A}(a_{2})-\fr_{A}(a_{2})\otimes\fl_{A}(a_{1}) \label{bia3}\\[-2mm]
&\qquad+\id\otimes\big(\fl_{A}(a_{1})\fl_{A}(a_{2})
-\fl_{A}(a_{2})\fl_{A}(a_{1})\big)-\big(\fl_{A}(a_{1})\fl_{A}(a_{2})
-\fl_{A}(a_{2})\fl_{A}(a_{1})\big)\otimes\id\Big)(r+\tau(r))=0,     \nonumber\\
&\Big(\fl_{A}(a)\otimes\id\otimes\id-\id\otimes\fl_{A}(a)\otimes\id\Big)
\Big(\tau(r)_{12}\diamond r_{13}+r_{12}\diamond r_{23}+r_{13}\diamond r_{23}
+r_{23}\diamond r_{13}\Big)  \label{bia4}\\[-2mm]
&\quad+\Big(\id\otimes\id\otimes(\fl_{A}+\fr_{A})(a)\Big)
\Big(r_{23}\diamond r_{13}-r_{13}\diamond r_{23}-(\id\otimes\id\otimes\id-\tau\otimes\id)
(r_{13}\diamond r_{12}+r_{12}\diamond r_{23})\Big)      \nonumber\\[-1mm]
&\qquad+\big((\id\otimes\fl_{A}(a))(r+\tau(r))\big)_{12}\diamond r_{23}
-\big((\fl_{A}(a_{1})\otimes\id)(r)\big)_{13}\diamond(r+\tau(r))_{12}=0,    \nonumber\\
&(\id\otimes\id\otimes\id-\id\otimes\tau)(\id\otimes\id\otimes(\fl_{A}+\fr_{A})(a))
(r_{13}\diamond\tau(r)_{23}-r_{12}\diamond r_{23}-r_{23}\diamond r_{12}
-r_{13}\diamond r_{12})=0. \label{bia5}
\end{align}
\end{pro}

\begin{defi}\label{def:dNPbi}
Let $(A, \diamond)$ be a Novikov algebra, $r=\sum_{i}x_{i}\otimes y_{i}\in A\otimes A$.
$$
\mathbf{N}_{r}:=r_{13}\diamond r_{23}+r_{12}\diamond r_{23}+r_{23}\diamond r_{12}
+r_{13}\diamond r_{12}=0
$$
is called the {\rm Novikov Yang-Baxter equation (or NYBE)} in $(A, \diamond)$,
where $r_{13}\diamond r_{23}=\sum_{i,j}x_{i}\otimes x_{j}\otimes(y_{i}\diamond y_{j})$,
$r_{12}\diamond r_{23}=\sum_{i,j}x_{i}\otimes(y_{i}\diamond x_{j})\otimes y_{j}$,
$r_{23}\diamond r_{12}=\sum_{i,j}x_{j}\otimes(x_{i}\diamond y_{j})\otimes y_{i}$ and
$r_{13}\diamond r_{12}=\sum_{i,j}(x_{i}\diamond x_{j})\otimes y_{j}\otimes y_{i}$.
\end{defi}

Thus, we have

\begin{pro}[\cite{HBG}]\label{pro:tri-bialg}
Let $(A, \diamond)$ be a Novikov algebra, $r\in A\otimes A$, and $\delta_{r}:
A\rightarrow A\otimes A$ be a linear map defined by Eq. \eqref{cobnov}.
If $r$ is a solution the NYBE in $(A, \diamond)$ and it is skew-symmetric,
i.e., $r=-\tau(r)$, then $(A, \diamond, \delta_{r})$ is a Novikov bialgebra,
which is called a {\rm triangular Novikov bialgebra} associated with $r$.
\end{pro}

Indeed, the skew-symmetry of solution $r$ can be weakened to the following invariance
of $\mathfrak{s}(r)$.

\begin{defi}\label{def:invar}
Let $(A, \diamond)$ be a Novikov algebra. An element $r\in A\otimes A$ is called
{\rm invariant} if for any $a\in A$,
$$
\Phi(a)(r):=\big(\fl_{A}(a)\otimes\id+\id\otimes(\fl_{A}+\fr_{A})(a)\big)(r)=0.
$$
\end{defi}

Let $(A, \diamond)$ be a Novikov algebra.
By direct calculation, for any $a_{1}, a_{2}\in A$, we have
\begin{align*}
&\;\Big(\id\otimes(\fl_{A}(a_{2}\diamond a_{1})+\fl_{A}(a_{1})\fl_{A}(a_{2}))
+(\fl_{A}+\fr_{A})(a_{1})\otimes(\fl_{A}+\fr_{A})(a_{2})\Big)(r+\tau(r))\\
=&\;\big(\id\otimes\fl_{A}(a_{2}\diamond a_{1})+\id\otimes\fl_{A}(a_{1})\fl_{A}(a_{2})\big)
(r+\tau(r))+\tau\big((\fl_{A}+\fr_{A})(a_{2})\otimes(\fl_{A}+\fr_{A})(a_{1})(r+\tau(r))\big)\\
=&\;\tau\Big(((\fl_{A}+\fr_{A})(a_{2})\otimes\id)\big((\fl_{A}(a_{1})\otimes\id
+\id\otimes(\fl_{A}+\fr_{A})(a_{1}))(r+\tau(r))\Big)\\[-1mm]
&\quad+\Big(\id\otimes\fl_{A}(a_{2}\diamond a_{1})+\id\otimes\fl_{A}(a_{1})\fl_{A}(a_{2})
-\id\otimes\fl_{A}(a_{2})\fl_{A}(a_{1})
-\id\otimes\fl_{A}(a_{1}\diamond a_{2})\Big)(r+\tau(r))\\
=&\;\tau\big(((\fl_{A}+\fr_{A})(a_{2})\otimes\id)(\Phi(a_{1})(r+\tau(r))\big).
\end{align*}
That is, Eq. \eqref{bia1} is equivalent to
\begin{align}
&((\fl_{A}+\fr_{A})(a_{2})\otimes\id)(\Phi(a_{1})(r+\tau(r))=0. \label{bia11}
\end{align}
Similarly, we also have Eq. \eqref{bia2} is equivalent to
\begin{align}
&\big((\fl_{A}+\fr_{A})(a_{1})\otimes\id\big)\big(\Phi(a_{2})(r+\tau(r))\big)
-\big((\fl_{A}+\fr_{A})(a_{2})\otimes\id\big)\big(\Phi(a_{1})(r+\tau(r))\big)=0, \label{bia12}
\end{align}
and Eq. \eqref{bia3} equivalent to
\begin{align}
&\big(\fr_{A}(a_{1})\otimes\id-\id\otimes\fr_{A}(a_{1})\big)\big(\Phi(a_{2})
(r+\tau(r))\big) \label{bia13}\\[-1mm]
&\qquad\qquad\qquad\qquad+\big(\id\otimes\fr_{A}(a_{2})-\fr_{A}(a_{2})\otimes\id\big)
\big(\Phi(a_{2})(r+\tau(r))\big)=0.  \nonumber
\end{align}

Let $(A, \diamond)$ be a Novikov algebra. For any $r=\sum_{i} x_{i}\otimes y_{i}
\in A\otimes A$, we define map $r^{\sharp}: A^{\ast}\rightarrow A$ by
$\langle\eta,\; r^{\sharp}(\xi)\rangle
=\langle\xi\otimes\eta,\; r\rangle$. Note that, for any $a\in A$ and $\xi, \eta\in A^{\ast}$,
\begin{align*}
0&=\langle\xi\otimes\eta,\; \big(\fl_{A}(a)\otimes\id+\id\otimes
(\fl_{A}+\fr_{A})(a)\big)(r)\rangle\\
&=-\langle\fl^{\ast}_{A}(a)(\xi)\otimes\eta+\xi\otimes\fl^{\ast}_{A}(a)(\eta)
+\xi\otimes\fr^{\ast}_{A}(a)(\eta),\ \ r\rangle\\
&=-\langle\eta,\ \ r^{\sharp}(\fl^{\ast}_{A}(a)(\xi))\rangle
-\langle\fl^{\ast}_{A}(a)(\eta),\ \ r^{\sharp}(\xi)\rangle
-\langle\fr^{\ast}_{A}(a)(\eta),\ \ r^{\sharp}(\xi)\rangle\\
&=\langle\eta,\ \ a\diamond r^{\sharp}(\xi)+r^{\sharp}(\xi)\diamond a
-r^{\sharp}(\fl^{\ast}_{A}(a)(\xi))\rangle,
\end{align*}
we get $\Phi(a)(r)=0$ if and only if for any $\xi\in A^{\ast}$,
$r^{\sharp}(\fl^{\ast}_{A}(a)(\xi))=a\diamond r^{\sharp}(\xi)+r^{\sharp}(\xi)\diamond a$.
Moreover, if $r$ is symmetric, i.e., $r=\tau(r)$, for any $\eta\in A^{\ast}$, we have
$r^{\sharp}(\fr^{\ast}_{A}(a)(\eta))=-r^{\sharp}(\eta)\diamond a$.
Let $A$ be vector space, for any $r\in A\otimes A$, it can be written as the sum of
symmetric part $\mathfrak{s}(r)$ and skew-symmetric part $\mathfrak{a}(r)$, i.e.,
$\mathfrak{s}(r)=\frac{1}{2}(r+\tau(r))$, $\mathfrak{a}(r)=\frac{1}{2}(r-\tau(r))
\in A\otimes A$ satisfying $r=\mathfrak{s}(r)+\mathfrak{a}(r)$.

\begin{pro}\label{pro:inv-bia}
Let $(A, \diamond)$ be a Novikov algebra, $r=\sum_{i}x_{i}\otimes y_{i}\in A\otimes A$,
and $\delta_{r}: A\rightarrow A\otimes A$ be a linear map defined by Eq. \eqref{cobnov}.
If $r$ is a solution of NYBE in $(A, \diamond)$ and the symmetric part $\mathfrak{s}(r)$
of $r$ is invariant, i.e., $\Phi(a)(\mathfrak{s}(r))=0$ for any $a\in A$, then
$(A, \diamond, \delta_{r})$ is a Novikov bialgebra, which is called a {\rm
quasi-triangular Novikov bialgebra} associated with $r$.
\end{pro}

\begin{proof}
If $r$ is a solution of NYBE in $(A, \diamond)$ and the symmetric part $\mathfrak{s}(r)$
of $r$ is invariant, then Eqs. \eqref{bia11}-\eqref{bia13} hold. That is, Eqs. \eqref{bia1}
-\eqref{bia3} hold.

For Eq. \eqref{bia4}, by direct calculation, we get
\begin{align*}
&\;(\id\otimes\delta_{r})\delta_{r}(a)-(\tau\otimes\id)(\id\otimes\delta_{r})\delta_{r}(a)
-(\delta_{r}\otimes\id)\delta_{r}(a)+(\tau\otimes\id)(\delta_{r}\otimes\id)\delta_{r}(a)\\
=&\;\big(\id\otimes\id\otimes(\fl_{A}+\fr_{A})(a)\big)\big((\tau\otimes\id)(\mathbf{N}_{r})
-\mathbf{N}_{r}\big)\\
&+\;\sum_{i,j} (a\diamond x_{i})\otimes(y_{i}\diamond x_{j})\otimes y_{j}
+(a\diamond x_{i})\otimes x_{j}\otimes(y_{i}\diamond y_{j})
+(a\diamond x_{i})\otimes x_{j}\otimes(y_{j}\diamond y_{i})\qquad\qquad\\[-4mm]
\end{align*}
\begin{align*}
&\qquad+x_{i}\otimes((y_{i}\diamond a)\diamond x_{j})\otimes y_{j}
-(y_{i}\diamond x_{j})\otimes(a\diamond x_{i})\otimes y_{j}
-x_{i}\otimes(a\diamond x_{j})\otimes(y_{j}\diamond y_{i})\\
&\qquad-x_{i}\otimes(a\diamond x_{j})\otimes(y_{i}\diamond y_{j})
-((y_{i}\diamond a)\diamond x_{j})\otimes x_{i}\otimes y_{j}
-((a\diamond x_{i})\diamond x_{j})\otimes y_{j}\otimes y_{i}\\
&\qquad-x_{i}\otimes(y_{i}\diamond(a\diamond x_{j}))\otimes y_{j}
+y_{i}\otimes((a\diamond x_{j})\diamond x_{i})\otimes y_{j}
+(y_{i}\diamond(a\diamond x_{j}))\otimes x_{i}\otimes y_{j}\\
=&\;\big(\id\otimes\id\otimes(\fl_{A}+\fr_{A})(a)\big)\big((\tau\otimes\id)(\mathbf{N}_{r})
-\mathbf{N}_{r}\big)+(\id\otimes\id\otimes\id-\tau\otimes\id)
(\fl_{A}(a)\otimes\id\otimes\id)(\mathbf{N}_{r})\\
&\qquad+(\tau\otimes\id-\id\otimes\id\otimes\id)(\id\otimes\fl_{A}(a)\otimes\id)
(\mathbf{N}_{r})\\[-1mm]
&\;+\sum_{j}\Big(\big(\id\otimes\fl_{A}(ax_{j})+\fr_{A}(ax_{j})\otimes\id
-\fr_{A}(x_{j})\fr_{A}(a)\otimes\id-\id\otimes\fl_{A}(a)\fr_{A}(x_{j})\big)
(r+\tau(r))\Big)\otimes y_{j}\\[-3mm]
&\qquad\quad+\sum_{j}\Big(\big(\id\otimes\fl_{A}(a)-\fl_{A}(a)\otimes\id\big)
\big(\Phi(a_{2})(r+\tau(r))\big)\Big)\otimes y_{j},
\end{align*}
for any $a\in A$, since
\begin{align*}
&\;\sum_{i,j} -(a\diamond x_{i})\otimes(x_{j}\diamond y_{i})\otimes y_{j}
-(a\diamond(x_{j}\diamond x_{i}))\otimes y_{i}\otimes y_{j}
-(a\diamond(y_{i}\diamond x_{j}))\otimes x_{i}\otimes y_{j}\\[-4mm]
&\qquad-(a\diamond(x_{j}\diamond y_{i}))\otimes x_{i}\otimes y_{j}
-(a\diamond y_{i})\otimes(x_{j}\diamond x_{i})\otimes y_{j}
+x_{i}\otimes((y_{i}\diamond a)\diamond x_{j})\otimes y_{j}\\
&\qquad+(x_{j}\diamond y_{i})\otimes(a\diamond x_{i})\otimes y_{j}
+y_{i}\otimes(a\diamond(x_{j}\diamond x_{i}))\otimes y_{j}
+x_{i}\otimes(a\diamond(y_{i}\diamond x_{j}))\otimes y_{j}\\
&\qquad+x_{i}\otimes(a\diamond(x_{j}\diamond y_{i}))\otimes y_{j}
+(x_{j}\diamond x_{i})\otimes(a\diamond y_{i})\otimes y_{j}
-((y_{i}\diamond a)\diamond x_{j})\otimes x_{i}\otimes y_{j}\\
&\qquad-((a\diamond x_{j})\diamond x_{i})\otimes y_{i}\otimes y_{j}
-x_{i}\otimes(y_{i}\diamond(a\diamond x_{j}))\otimes y_{j}
+y_{i}\otimes((a\diamond x_{j})\diamond x_{i})\otimes y_{j}\\
&\qquad+(y_{i}\diamond(a\diamond x_{j}))\otimes x_{i}\otimes y_{j}\\
=&\;\sum_{j}\Big(\big(\id\otimes\fl_{A}(ax_{j})+\fr_{A}(ax_{j})\otimes\id
-\fr_{A}(x_{j})\fr_{A}(a)\otimes\id-\id\otimes\fl_{A}(a)\fr_{A}(x_{j})\big)
(r+\tau(r))\Big)\otimes y_{j}\\[-4mm]
&\qquad+\sum_{j}\Big(\big(\id\otimes\fl_{A}(a)-\fl_{A}(a)\otimes\id\big)
\big(\Phi(a_{2})(r+\tau(r))\big)\Big)\otimes y_{j}.
\end{align*}
Note that, for any $a\in A$ and $\xi, \eta\in A^{\ast}$,
\begin{align*}
&\;\langle\xi\otimes\eta,\ \ \Big(\id\otimes\fl_{A}(a\diamond x_{j})
+\fr_{A}(a\diamond x_{j})\otimes\id-\fr_{A}(x_{j})\fr_{A}(a)\otimes\id
-\id\otimes\fl_{A}(a)\fr_{A}(x_{j})\Big)(\hat{r})\rangle\\
=&\;-\langle\xi\otimes\fl_{A}^{\ast}(a\diamond x_{j})(\eta)
+\fr_{A}^{\ast}(a\diamond x_{j})(\xi)\otimes\eta
+\fr_{A}^{\ast}(a)(\fr_{A}^{\ast}(x_{j})(\xi))\otimes\eta
+\xi\otimes\fr_{A}^{\ast}(x_{j})(\fl_{A}^{\ast}(a)(\eta)),\ \ \hat{r}\rangle\\
=&\;-\langle\fl_{A}^{\ast}(a\diamond x_{j})(\eta),\ \ \hat{r}^{\sharp}(\xi)\rangle
-\langle\eta,\ \ \hat{r}^{\sharp}(\fr_{A}^{\ast}(a\diamond x_{j})(\xi))\rangle\\[-1mm]
&\qquad-\langle\eta,\ \ \hat{r}^{\sharp}(\fr_{A}^{\ast}(a)(\fr_{A}^{\ast}(x_{j})(\xi)))\rangle
-\langle\fr_{A}^{\ast}(x_{j})(\fl_{A}^{\ast}(a)(\eta)),\ \hat{r}^{\sharp}(\xi)\rangle\\
=&\;\langle\eta,\ \ (a\diamond x_{j})\diamond\hat{r}^{\sharp}(\xi)
+\hat{r}^{\sharp}(\xi)(a\diamond x_{j})-(\hat{r}^{\sharp}(\xi)\diamond x_{j})\diamond a
-a\diamond(\hat{r}^{\sharp}(\xi)\diamond x_{j})\rangle\\
=&\;0,
\end{align*}
we get Eq. \eqref{bia4} holds in this case. Similarly, we also have Eq. \eqref{bia5} holds.
Thus, by Proposition \ref{pro:cob-nov}, $(A, \diamond, \delta_{r})$ is a
Novikov bialgebra.
\end{proof}

Obviously, the triangular Novikov bialgebra is a subclass of quasi-triangular
Novikov bialgebra, an important subclass of quasi-triangular Novikov bialgebra.
In the next section, we will give another important subclass of quasi-triangular
Novikov bialgebra. Here, we give a quasi-triangular Novikov bialgebra from
any Novikov bialgebra.

\begin{pro}\label{pro:double}
Let $(A, \diamond, \delta)$ be a Novikov bialgebra. Then there is a quasi-triangular Novikov
bialgebra structure on $A \oplus A^{\ast}$, where the Novikov bialgebra structure on
$A\oplus A^{\ast}$ is $(\ast, \delta_{\tilde{r}})$, $``\ast"$ is given by Eq.
\eqref{product} and $``\delta_{\tilde{r}}"$ is given by Eq. \eqref{cobnov}, for some
$\tilde{r}\in(A\oplus A^{\ast})\otimes(A\oplus A^{\ast})$. The quasi-triangular Novikov
bialgebra $(A\oplus A^{\ast}, \ast, \delta_{\tilde{r}})$ is called the {\rm Novikov
double of Novikov bialgebra $(A, \diamond, \delta)$}.
\end{pro}

\begin{proof}
Let $\tilde{r}\in A \otimes A^{\ast}\subset (A \oplus A^{\ast})\otimes(A \oplus A^{\ast})$
corresponds to the identity map $\id: A\rightarrow A$. That is, $\tilde{r}=\sum_{i=1}^{n}
e_{i}\otimes f_{i}$, where $\{e_{1}, e_{2}, \cdots, e_{n}\}$ is a basis of $A$ and
$\{f_{1}, f_{2}, \cdots, f_{n}\}$ be its dual basis in $A^{\ast}$.
Since $(A, \diamond, \delta)$ is a Novikov bialgebra, there is a Novikov
algebra structure $``\diamond_{\delta}"$ on $A^{\ast}$, and a Novikov algebra structure
$``\ast"$ on $A \oplus A^{\ast}$ which is given by $(a_{1}, \xi_{1})\ast(a_{2}, \xi_{2})
:=\big(a_{1}\diamond y+\fl_{A^{\ast}}^{\ast}(\xi_{1})(a_{2})
+\fr_{A^{\ast}}^{\ast}(\xi_{1})(a_{2})-\fr_{A^{\ast}}^{\ast}(\xi_{2})(a_{1}),\ \
\xi_{1}\diamond_{\delta}\xi_{2}+\fl_{A}^{\ast}(a_{1})(\xi_{2})+\fr_{A}^{\ast}(a_{1})(\xi_{2})
-\fr_{A}^{\ast}(a_{2})(\xi_{1})\big)$, for any $a_{1}, a_{2}\in A$,
$\xi_{1}, \xi_{2}\in A^{\ast}$.
We will show that there is a quasi-triangular Novikov bialgebra structure on
Novikov algebra $(A\oplus A^{\ast}, \ast)$. First, for any $(e_{s}, f_{t})$,
$(e_{k}, f_{l})$, $(e_{p}, f_{q})\in A\oplus A^{\ast}$,
\begin{align*}
&\; \langle\tilde{r}_{13}\ast\tilde{r}_{23}+\tilde{r}_{12}\ast\tilde{r}_{23}
+\tilde{r}_{23}\ast\tilde{r}_{12}+\tilde{r}_{13}\ast\tilde{r}_{12},\ \
(e_{s}, f_{t})\otimes(e_{k}, f_{l})\otimes(e_{p}, f_{q})\rangle \\
=&\;\sum_{i,j}\langle e_{i}\otimes e_{j}\otimes f_{i}\ast f_{j}
+e_{i}\otimes f_{i}\ast e_{j}\otimes f_{j}+e_{j}\otimes e_{i}\ast f_{j}\otimes f_{i}
+e_{i}\ast e_{j}\otimes f_{j}\otimes f_{i}, \\[-4mm]
&\qquad\qquad\qquad\qquad\qquad\qquad\qquad\qquad\qquad\qquad\qquad (e_{s}, f_{t})
\otimes(e_{k}, f_{l})\otimes(e_{p}, f_{q})\rangle \\
=&\;\sum_{i,j}\Big(\langle e_{i}, f_{t}\rangle\langle e_{j}, f_{l}\rangle
\langle f_{i}\diamond_{\delta}f_{j},\; e_{p}\rangle-\langle e_{i}, f_{t}\rangle
\langle e_{j},\; f_{i}\diamond_{\delta}f_{l}\rangle\langle f_{j}, e_{p}\rangle
-\langle e_{i}, f_{t}\rangle\langle e_{j},\; f_{l}\diamond_{\delta}f_{i}\rangle
\langle f_{j}, e_{p}\rangle\\[-4mm]
&\qquad\quad+\langle e_{i}, f_{t}\rangle\langle f_{i}, e_{k}\diamond e_{j}\rangle
\langle f_{j}, e_{p}\rangle
-\langle e_{j}, f_{t}\rangle\langle f_{j}, e_{i}\diamond e_{k}\rangle
\langle f_{i}, e_{p}\rangle
-\langle e_{j}, f_{t}\rangle\langle f_{j}, e_{k}\diamond e_{i}\rangle
\langle f_{i}, e_{p}\rangle\\[-1mm]
&\qquad\quad+\langle e_{j}, f_{t}\rangle\langle e_{i},\; f_{l}\diamond_{\delta}f_{j}\rangle
\langle f_{i}, e_{p}\rangle+\langle e_{i}\diamond e_{j}, f_{t}\rangle
\langle f_{j}, e_{k}\rangle\langle f_{i}, e_{p}\rangle\Big) \\
=&\;\langle f_{t}\diamond_{\delta}f_{l},\; e_{p}\rangle
-\langle e_{p},\; f_{t}\diamond_{\delta}f_{l}\rangle
-\langle e_{p},\; f_{l}\diamond_{\delta}f_{t}\rangle
+\langle f_{t}, e_{k}\diamond e_{p} \rangle\\[-1mm]
&\qquad-\langle f_{t}, e_{p}\diamond e_{k}\rangle-\langle f_{t}, e_{k}\diamond e_{p}\rangle
+\langle e_{p},\; f_{l}\diamond_{\delta}f_{t}\rangle
+\langle e_{p}\diamond e_{k}, f_{t}\rangle\\
=&\;0.
\end{align*}
That is, $\tilde{r}$ is a solution of NYBE in $(A\oplus A^{\ast}, \ast)$.
Second, for any $e_{k}\in A\oplus A^{\ast}$, $k=1, 2, \cdots, n$, we have
\begin{align*}
&\;\big(\fl_{A\oplus A^{\ast}}(e_{k})\otimes\id+\id\otimes(\fl_{A\oplus A^{\ast}}
+\fr_{A\oplus A^{\ast}})(e_{k})\big)(\tilde{r}+\tau(\tilde{r})) \\
=&\;\sum_{i}(e_{k}\diamond e_{i})\otimes f_{i}+(e_{k}\ast f_{i})\otimes e_{i}
+e_{i}\otimes(e_{k}\ast f_{i})\\[-4mm]
&\qquad+e_{i}\otimes(f_{i}\ast e_{k})+f_{i}\otimes(e_{k}\diamond e_{i})
+f_{i}\otimes(e_{i}\diamond e_{k}) \\
=&\;\sum_{i,l}(e_{k}\diamond e_{i})\otimes f_{i}
-\langle f_{i},\; e_{k}\diamond e_{l}\rangle f_{l}\otimes e_{i}
-\langle f_{i},\; e_{l}\diamond e_{k}\rangle f_{l}\otimes e_{i}
+\langle e_{k},\; f_{l}\diamond_{\delta}f_{i}\rangle e_{l}\otimes e_{i}\\[-4mm]
&\qquad-e_{i}\otimes\langle f_{i},\; e_{k}\diamond e_{l}\rangle f_{l}
-e_{i}\otimes\langle f_{i},\; e_{l}\diamond e_{k}\rangle f_{l}
+e_{i}\otimes\langle e_{k},\; f_{l}\diamond_{\delta}f_{i}\rangle e_{l}
-e_{i}\otimes\langle e_{k},\; f_{i}\diamond_{\delta}f_{l}\rangle e_{l}\\
&\qquad-e_{i}\otimes\langle e_{k},\; f_{l}\diamond_{\delta}f_{i}\rangle e_{l}
+e_{i}\otimes\langle f_{i},\; e_{l}\diamond e_{k}\rangle f_{l}
+f_{i}\otimes(e_{k}\diamond e_{i})+f_{i}\otimes(e_{i}\diamond e_{k})\\
=&\;\sum_{i,l} (e_{k}\diamond e_{i})\otimes f_{i}-f_{l}\otimes(e_{k}\diamond e_{l})
-f_{l}\otimes(e_{l}\diamond e_{k})
+\langle e_{k},\; f_{l}\diamond_{\delta}f_{i}\rangle e_{l}\otimes e_{i}\\[-4mm]
&\qquad-(e_{k}\diamond e_{l})\otimes f_{l}-(e_{l}\diamond e_{k})\otimes f_{l}
+e_{i}\otimes\langle e_{k},\; f_{l}\diamond_{\delta}f_{i}\rangle e_{l}
-e_{i}\otimes\langle e_{k},\; f_{i}\diamond_{\delta}f_{l}\rangle e_{l}\\
&\qquad-e_{i}\otimes\langle e_{k},\; f_{l}\diamond_{\delta}f_{i}\rangle e_{l}
+(e_{l}\diamond e_{k})\otimes f_{l}+f_{i}\otimes(e_{k}\diamond e_{i})
+f_{i}\otimes(e_{i}\diamond e_{k})\\
=&\; 0.
\end{align*}
Similarly, this equation is also holds for $f_{k}\in A\oplus A^{\ast}$.
Hence, $\frac{1}{2}(\tilde{r}+\tau(\tilde{r}))$ is invariant. Therefore,
the Novikov bialgebra $(A\oplus A^{\ast}, \ast, \delta_{\tilde{r}})$ is a
quasi-triangular Novikov bialgebra.
\end{proof}

\bigskip
\section{Factorizable Novikov bialgebras}\label{sec:fact}
In this section, we establish the factorizable theories for Novikov bialgebras.
A factorizable Novikov bialgebra induces a factorization of the underlying Novikov
algebra and the double of any Novikov bialgebra naturally admits a factorizable Novikov
bialgebra structure. Moreover, we introduce the notion of quadratic Rota-Baxter Novikov
algebras and show that there is an one-to-one correspondence between factorizable
Novikov bialgebras and quadratic Rota-Baxter Novikov algebras of nonzero weights.
For any $r\in A\otimes A$, we have defined a linear map $r^{\sharp}: A^{\ast}
\rightarrow A$ by $\langle r^{\sharp}(\xi_{1}),\; \xi_{2}\rangle
=\langle\xi_{1}\otimes\xi_{2},\; r\rangle$. Now we define another linear map
$r^{\natural}:A^{\ast}\rightarrow A$ by
$$
\langle\xi_{1},\;r^{\natural}(\xi_{2})\rangle=-\langle\xi_{1}\otimes\xi_{2},\;r\rangle,
$$
for any $\xi_{1}, \xi_{2}\in A^{\ast}$. If $(A, \cdot, \delta_{r})$ is a Novikov bialgebra,
then the Novikov algebra structure $\diamond_{r}$ on $A^{\ast}$ dual to the comultiplication
$\delta_{r}$ defined by Eq. \eqref{cobnov} is given by
$$
\xi\diamond_{r}\eta=(\fl_{A}^{\ast}+\fr_{A}^{\ast})(r^{\sharp}(\xi))(\eta)
-\fr_{A}^{\ast}(r^{\natural}(\eta))(\xi),
$$
for any $\xi, \eta\in A^{\ast}$. Denote by $\mathcal{I}$ the operator
$$
\mathcal{I}=r^{\sharp}-r^{\natural}:\quad A^{\ast}\longrightarrow A.
$$
Then $\mathcal{I}^{\ast}=\mathcal{I}$ if we identify $(A^{\ast})^{\ast}$ with $A$.
Actually, $\frac{1}{2}\mathcal{I}$ is the contraction of the symmetric part $\mathfrak{s}(r)$
of $r$, which means that $\frac{1}{2}\langle\mathcal{I}(\xi),\; \eta\rangle
=\mathfrak{s}(r)(\xi, \eta)$. If $r$ is skew-symmetric, then $\mathcal{I}=0$.
Now we give another characterization of the invariant condition for $r
+\tau(r)\in A\otimes A$.

\begin{lem}\label{lem:syminva}
Let $(A, \diamond)$ be a Novikov algebra and $r\in A\otimes A$. The symmetric part
$\mathfrak{s}(r)$ of $r$ is invariant if and only if $\mathcal{I}\fl_{A}^{\ast}(a)=
\fl_{A}(a)\mathcal{I}+\fr_{A}(a)\mathcal{I}$, if and only if
$\fl_{A}(a)\mathcal{I}=\mathcal{I}\fl_{A}^{\ast}(a)+\mathcal{I}\fr_{A}^{\ast}(a)$,
for any $a\in A$, where $\mathcal{I}=r^{\sharp}-r^{\natural}: A^{\ast}\rightarrow A$.
\end{lem}

\begin{proof}
According to the analysis above Proposition \ref{pro:inv-bia}, it can be
concluded that the symmetric part $\mathfrak{s}(r)$ of $r$ is invariant if and only if
$$
\mathfrak{s}(r)^{\sharp}(\fl_{A}^{\ast}(a)(\xi))
=a\diamond\mathfrak{s}(r)^{\sharp}(\xi)+\mathfrak{s}(r)^{\sharp}(\xi)\diamond a,
$$
for all $a\in A,\xi\in A^{\ast}$. Since $\mathfrak{s}(r)^{\sharp}=\frac{1}{2}\mathcal{I}$,
we get that the symmetric part $\mathfrak{s}(r)$ is invariant if and only if
$\mathcal{I}\fl_{A}^{\ast}(a)=\fl_{A}(a)\mathcal{I}+\fr_{A}(a)\mathcal{I}$
for any $a\in A$. Moreover, since $\mathcal{I}^{\ast}=\mathcal{I}$, we have
$\mathcal{I}\fl_{A}^{\ast}(a)=\fl_{A}(a)\mathcal{I}+\fr_{A}(a)\mathcal{I}$ is equivalent
to $\fl_{A}(a)\mathcal{I}=\mathcal{I}\fl_{A}^{\ast}(a)+\mathcal{I}\fr_{A}^{\ast}(a)$.
The proof is finished.
\end{proof}

If the symmetric part $\mathfrak{s}(r)$ of $r\in A\otimes A$ is invariant,
we can also give a characterization for solution of the NYBE in a Novikov algebra.

\begin{pro}\label{pro:solu-equiv}
Let $(A, \diamond)$ be a Novikov algebra and $r\in A\otimes A$.
Assume that the symmetric part $\mathfrak{s}(r)$ of $r$ is invariant.
Then $r$ is a solution of the NYBE in $(A, \diamond)$ if and only if
$(A^{\ast}, \diamond_{r})$ is a Novikov algebra and the linear maps
$r^{\sharp}, r^{\natural}: (A^{\ast}, \diamond_{r})\rightarrow (A, \diamond)$
are Novikov algebra homomorphisms.
\end{pro}

\begin{proof}
For any $\xi, \eta, \zeta\in A^{\ast}$, we have
\begin{align*}
&\;\langle r^{\sharp}(\xi\diamond_{r}\eta)
-r^{\sharp}(\xi)\diamond r^{\sharp}(\eta),\ \ \zeta\rangle\\
=&\;\langle r^{\sharp}((\fl_{A}^{\ast}+\fr_{A}^{\ast})(r^{\sharp}(\xi))(\eta))
-r^{\sharp}(\fr_{A}^{\ast}(r^{\natural}(\eta))(\xi)),\ \ \zeta\rangle
-\langle r^{\sharp}(\xi)\diamond r^{\sharp}(\eta),\ \ \zeta\rangle\\
=&\;\langle r^{\sharp}(\xi)\diamond r^{\natural}(\zeta),\ \ \eta\rangle
+\langle r^{\natural}(\zeta)\diamond r^{\sharp}(\xi),\ \ \eta\rangle
-\langle r^{\natural}(\zeta)\diamond r^{\natural}(\eta),\ \ \xi\rangle
-\langle r^{\sharp}(\xi)\diamond r^{\sharp}(\eta),\ \ \zeta\rangle\\
=&\;-\langle\xi\otimes\eta\otimes\zeta,\ \ \mathbf{N}_{r}\rangle.
\end{align*}
Thus, we get that $r^{\sharp}: (A^{\ast}, \diamond_{r})\rightarrow(A, \diamond)$
is Novikov algebra homomorphism if $r$ is a solution of the NYBE in $(A, \diamond)$.
The situation with $r^{\natural}$ is similar. Conversely, it is directly available
from the above calculation.
\end{proof}

\begin{defi}\label{def:fact}
Let $(A, \diamond)$ be a Novikov algebra, $r\in A\otimes A$, and $(A, \diamond, \delta_{r})$
be the quasi-triangular Novikov bialgebra associated with $r$. If
$\mathcal{I}=r^{\sharp}-r^{\natural}: A^{\ast}\rightarrow A$ is a linear isomorphism
of vector spaces, then $(A, \diamond, \delta_{r})$ is called a {\rm factorizable
Novikov bialgebra}.
\end{defi}

Obviously, the factorizable Novikov bialgebra is a subclass of quasi-triangular
Novikov bialgebra, and we can view the factorizable Novikov bialgebra
is the opposite of the triangular Novikov bialgebra, since $\mathcal{I}=0$
for triangular Novikov bialgebra. We consider the linear map $\mathcal{I}$ as a
composition of maps as follows:
$$
A^{\ast}\xrightarrow{\quad r^{\sharp}\oplus r^{\natural}\quad}
A\oplus A\xrightarrow{\quad(x,y)\mapsto x-y\quad} A.
$$
The following result justifies the terminology of a factorizable Novikov bialgebra.

\begin{pro}\label{pro:fact}
Let $(A, \diamond)$ be a Novikov algebra and $r\in A\otimes A$. Assume the Novikov bialgebra
$(A, \diamond, \delta_{r})$ is factorizable. Then $\Img(r^{\sharp}\oplus r^{\natural})$
is a Novikov subalgebra of the direct sum Novikov algebra $(A\oplus A, \ast)$, which is
isomorphism to the Novikov algebra $(A^{\ast}, \diamond_{r})$.
Moreover, any $a\in A$ has an unique decomposition $a=a_{+}+a_{-}$,
where $a_{+}\in\Img(r^{\sharp})$ and $a_{-}\in\Img(r^{\natural})$.
\end{pro}

\begin{proof}
Since $(A, \diamond, \delta_{r})$ is a quasi-triangular Novikov bialgebra,
both $r^{\sharp}$ and $r^{\natural}$ are Novikov algebra homomorphisms.
Therefore, $\Img(r^{\sharp}\oplus r^{\natural})$ is a Novikov subalgebra of the
direct sum Novikov algebra $(A\oplus A, \ast)$. Since $\mathcal{I}: A^{\ast}\rightarrow A$
is a linear isomorphism, it follows that $r^{\sharp}\oplus r^{\natural}$ is injective,
then the Novikov algebra $\Img(r^{\sharp}\oplus r^{\natural})$ is isomorphic
to the Novikov algebra $(A^{\ast}, \diamond_{r})$.
Moreover, since $\mathcal{I}$ is an isomorphism, for any $a\in A$, we have
$$
a=(r^{\sharp}-r^{\natural})(\mathcal{I}^{-1}(a))
=r^{\sharp}(\mathcal{I}^{-1}(a))-r^{\natural}(\mathcal{I}^{-1}(a)),
$$
which implies that $a=a_{+}+a_{-}$, where $a_{+}=r^{\sharp}(\mathcal{I}^{-1}(a))$ and
$a_{-}=-r^{\natural}(\mathcal{I}^{-1}(a))$. The uniqueness also follows from
the fact that $\mathcal{I}$ is an isomorphism.
\end{proof}

In Proposition \ref{pro:double}, we construct a quasi-triangular Novikov bialgebra
from any Novikov bialgebra. Here, we show that the quasi-triangular Novikov bialgebra
constructed in Proposition \ref{pro:double} is factorizable.

\begin{pro}\label{pro:doufact}
Let $(A, \diamond, \delta)$ be a Novikov bialgebra. Using the notation from
Proposition \ref{pro:double}, we get that the Novikov double $
(A\oplus A^{\ast}, \ast, \delta_{\tilde{r}})$ is factorizable.
\end{pro}

\begin{proof}
By Proposition \ref{pro:double}, the Novikov bialgebra $(A\oplus A^{\ast},
\ast, \delta_{\tilde{r}})$ is quasi-triangular. Note that $\tilde{r}^{\sharp},
\tilde{r}^{\natural}: (A\oplus A^{\ast})^{\ast}=A^{\ast}\oplus A\rightarrow
A\oplus A^{\ast}$ are given by
\begin{align*}
\tilde{r}^{\sharp}(\xi,\; a)=(0,\; \xi),\qquad and\qquad
\tilde{r}^{\natural}(\xi,\; a)=(-a,\; 0),
\end{align*}
for any $a\in A$ and $\xi\in A^{\ast}$. This implies that $\mathcal{I}(\xi, a)=(a, \xi)$.
Thus $\mathcal{I}$ is a linear isomorphism, and so that, $(A\oplus A^{\ast}, \ast,
\delta_{\tilde{r}})$ is a factorizable Novikov bialgebra.
\end{proof}

\begin{ex}\label{ex:fact-novbia}
Consider the 2-dimensional Novikov bialgebra $(A, \cdot, \Delta)$ given
by Example \ref{ex:2bialg}, i.e., $e_{1}\diamond e_{1}=e_{2}$ and
$\Delta(e_{1})=e_{2}\otimes e_{2}$. Denote $\{f_{1}, f_{2}\}$ be the dual basis
of $\{e_{1}, e_{2}\}$ and $\tilde{r}=e_{1}\otimes f_{1}+e_{2}\otimes f_2\in
A\otimes A^{\ast}\subset(A\oplus A^{\ast})\otimes(A\oplus A^{\ast})$.
Then, by Proposition \ref{pro:double}, we get a quasi-triangular Novikov bialgebra
$(A\oplus A^{\ast}, \ast, \delta_{\tilde{r}})$, where
\begin{align*}
e_{1}\ast e_{1}=e_{2},\qquad &f_{2}\ast f_{2}=f_{1},\qquad e_{1}\ast f_{2}
=-2f_{1}+e_{2},\qquad f_2\ast e_{1}=-2e_{2}+f_1,\\
&\delta_{\tilde{r}}(e_{1})=e_{2}\otimes e_{2},\qquad
\delta_{\tilde{r}}(f_{2})=-f_{1}\otimes f_{1},
\end{align*}
since $\tilde{r}+\tau(\tilde{r})=e_{1}\otimes f_{1}+e_{2}\otimes f_{2}+f_{1}\otimes e_{1}
+f_{2}\otimes e_{2}$ is invariant. By direct calculation,we get that the linear map
$\mathcal{I}: (A\oplus A^{\ast})^{\ast} \rightarrow(A\oplus A^{\ast})$ is given by
\begin{align*}
\mathcal{I}(e_{1}^{\ast})=f_{1},\qquad \mathcal{I}(e_{2}^{\ast})=f_{2},\qquad
\mathcal{I}(f_{1}^{\ast})=e_{1},\qquad\mathcal{I}(f_{2}^{\ast})=e_{2},
\end{align*}
where $\{e_{1}^{\ast}, e_{2}^{\ast}, f_{1}^{\ast}, f_{2}^{\ast}\}$ in
$(A\oplus A^{\ast})^{\ast}$ is a dual basis of $\{e_{1}, e_{2}, f_{1}, f_{2}\}$ in
$A\oplus A^{\ast}$. Hence, we get $\mathcal{I}$ is a linear isomorphism, and so that
$(A\oplus A^{\ast}, \ast, \delta_{\tilde{r}})$ is a factorizable Novikov bialgebra.
\end{ex}

Next, we will give a characterization of factorizable Novikov bialgebras by the
Rota-Baxter operator on quadratic Novikov algebras. Recall that a {\it quadratic
Novikov algebra} $(A, \diamond, \mathfrak{B})$ is a Novikov algebra
$(A, \diamond)$ equipped with a nondegenerate symmetric invariant bilinear form
$\mathfrak{B}(-,-)$. Let $(A, \diamond)$ be a Novikov algebra and $\lambda\in\Bbbk$.
A linear map $P: A\rightarrow A$ is called a {\it Rota-Baxter operator} of weight
$\lambda$ on $(A, \diamond)$ if
$$
P(a_{1})\diamond P(a_{2})=P(P(a_{1})\diamond a_{2}+a_{1}\diamond P(a_{2})
+\lambda a_{1}\diamond a_{2}),
$$
for any $a_{1}, a_{2}\in A$. In this case, we called $(A, \diamond, P)$ is a
Rota-Baxter Novikov algebra of weight $\lambda$.
For any Rota-Baxter Novikov algebra $(A, \diamond, P)$ of weight $\lambda$, there is
a new multiplication $\diamond_{P}$ on $A$ defined by
$$
a_{1}\diamond_{P}a_{2}=P(a_{1})\diamond a_{2}+a_{1}\diamond P(a_{2})
+\lambda a_{1}\diamond a_{2},
$$
for any $a_{1}, a_{2}\in A$. The Novikov algebra $(A, \diamond_{P})$ is called the
{\it descendent Novikov algebra} and denoted by $A_{P}$. Furthermore, $P$ is a Novikov
algebra homomorphism from the Novikov algebra $(A, \diamond_{P})$ to $(A, \diamond)$.

\begin{defi}\label{def:quadr}
Let $(A, \diamond, P)$ be a Rota-Baxter Novikov algebra of weight $\lambda$, and
$\mathfrak{B}(-,-)$ be a nondegenerate symmetric bilinear form on $(A, \diamond)$.
The triple $(A, P, \mathfrak{B})$ is called a quadratic Rota-Baxter Novikov algebra
of weight $\lambda$ if $(A, \diamond, \mathfrak{B})$ is a quadratic Novikov algebra
and the following compatibility condition holds:
\begin{align}
\mathfrak{B}(P(a_{1}),\; a_{2})+\mathfrak{B}(a_{1},\; P(a_{2}))
+\lambda\mathfrak{B}(a_{1}, a_{2})=0,                \label{quaRB}
\end{align}
for any $a_{1}, a_{2}\in A$.
\end{defi}

\begin{thm}\label{thm:quafact}
Let $(A, \diamond)$ be a Novikov algebra and $r\in A\otimes A$. Suppose the Novikov
bialgebra $(A, \diamond, \delta_{r})$ is factorizable and $\mathcal{I}=r^{\sharp}
-r^{\natural}: A^{\ast}\rightarrow A$.
We define a bilinear form $\mathfrak{B}_{\mathcal{I}}(-,-)$ by $\mathfrak{B}_{\mathcal{I}}
(a_{1}, a_{2})=\langle\mathcal{I}^{-1}(a_{1}),\; a_{2}\rangle$, for any $a_{1}, a_{2}\in A$.
Then $(A, \diamond, \mathfrak{B}_{\mathcal{I}})$ is a quadratic Novikov algebra.
Moreover, the linear map $P=\lambda r^{\natural}\mathcal{I}^{-1}: A\rightarrow A$ is a
Rota-Baxter operator of weight $\lambda$ on $(A, \diamond, \mathfrak{B}_{\mathcal{I}})$.
	
Conversely, for any quadratic Rota-Baxter Novikov algebra $(A, P, \mathfrak{B})$
of weight $\lambda$, we have a linear isomorphism $\mathcal{I}_{\mathfrak{B}}:
A^{\ast}\rightarrow A$ by $\langle\mathcal{I}^{-1}_{\mathfrak{B}}(a_{1}),\; a_{2}\rangle
:=\mathfrak{B}(a_{1}, a_{2})$, for any $a_{1}, a_{2}\in A$.
If $\lambda\neq0$, we define
$$
r^{\sharp}:=\mbox{$\frac{1}{\lambda}$}(P+\lambda\id)
\mathcal{I}_{\mathfrak{B}}:\quad A^{\ast}\longrightarrow A,
$$
and define $r\in A\otimes A$ by $\langle r^{\sharp}(\xi),\; \eta\rangle
=\langle r,\; \xi\otimes\eta\rangle$, Then, $r$ is a solution of the NYBE
in $(A, \diamond)$ and induces a factorizable Novikov bialgebra
$(A, \diamond, \delta_{r})$, where $\delta_{r}$ is given by Eq. \eqref{cobnov}.
\end{thm}

\begin{proof}
Suppose the Novikov bialgebra $(A, \diamond, \delta_{r})$ is factorizable,
then, $r^{\sharp}, r^{\natural}: (A_{r}^{\ast}, \diamond_{r})\rightarrow(A, \diamond)$
are both Novikov algebra homomorphisms, for all $a_{1}, a_{2}, a_{3}\in A$. Thus, we have
\begin{equation}
\begin{split}
\mathcal{I}(\mathcal{I}^{-1}(a_{1})\diamond_{r}\mathcal{I}^{-1}(a_{2}))
&=(r^{\sharp}-r^{\natural})(\mathcal{I}^{-1}(a_{1})\diamond_{r}\mathcal{I}^{-1}(a_{2}))\\
&=(\mathcal{I}+r^{\natural})(\mathcal{I}^{-1}(a_{1}))\diamond(\mathcal{I}+r^{\natural})
(\mathcal{I}^{-1}(a_{2}))-r^{\natural}(\mathcal{I}^{-1}(a_{1}))\diamond
r^{\natural}(\mathcal{I}^{-1}(a_{2}))\\
&=r^{\natural}(\mathcal{I}^{-1}(a_{1}))\diamond a_{2}
+a_{1}\diamond r^{\natural}(\mathcal{I}^{-1}(a_{2}))+a_{1}\diamond a_{2}.   \label{I1}
\end{split}
\end{equation}
Therefore,
\begin{align*}
P(P(a_{1})\diamond a_{2}+a_{1}\diamond P(a_{2})+\lambda a_{1}\diamond a_{2})
&=\lambda^{2}r^{\natural}\Big(\mathcal{I}^{-1}\Big(r^{\natural}(\mathcal{I}^{-1}(a_{1}))
\diamond a_{2}+a_{1}\diamond r^{\natural}(\mathcal{I}^{-1}(a_{2}))
+a_{1}\diamond a_{2}\Big)\Big)\\
&=\lambda^{2}r^{\natural}\big(\mathcal{I}^{-1}(a_{1})\diamond_{r}\mathcal{I}^{-1}(a_{2})\big)\\
&=\lambda^2\big(r^{\natural}(\mathcal{I}^{-1}(a_{1}))\diamond
r^{\natural}(\mathcal{I}^{-1}(a_{2}))\big)\\
&=P(a_{1})\diamond P(a_{2}),
\end{align*}
which implies that $P$ is a Rota-Baxter operator of weight $\lambda$ on $(A, \diamond)$.
	
Next, we prove that $(A, \diamond_{P}, \mathfrak{B}_{\mathcal{I}})$ is a quadratic
Novikov algebra. Since $\mathcal{I}^{\ast}=\mathcal{I}$ if we identify $(A^{\ast})^{\ast}$
with $A$, we get $\mathfrak{B}_{\mathcal{I}}(a_{1}, a_{2})=\langle\mathcal{I}^{-1}
(a_{1}),\; a_{2}\rangle=\langle a_{1},\; \mathcal{I}^{-1}(a_{2})\rangle=
\mathfrak{B}_{\mathcal{I}}(a_{2}, a_{1})$. That is, $\mathfrak{B}_\mathcal{I}$ is symmetric.
Moreover, since the symmetric part $\mathfrak{s}(r)$ of $r$ is invariant, by Lemma
\ref{lem:syminva}, we have $\mathcal{I}\fl_{A}^{\ast}(a)=\fl_{A}(a)\mathcal{I}
+\fr_{A}(a)\mathcal{I}$ for $a\in A$. Thus, for any $a_{1}, a_{2}\in A$,
\begin{align*}
\mathfrak{B}_{\mathcal{I}}(a_{1}\diamond a_{2},\; a_{3})
+\mathfrak{B}_{\mathcal{I}}(a_{2},\; a_{1}\diamond a_{3}+a_{3}\diamond a_{1})
&=\langle\mathcal{I}^{-1}(a_{1}\diamond a_{2}),\; a_{3}\rangle
+\langle\mathcal{I}^{-1}(a_{2}),\; a_{1}\diamond a_{3}+a_{3}\diamond a_{1}\rangle\\
&=\langle\mathcal{I}^{-1}(\fl_{A}(a_{1}))(a_{2}),\; a_{3}\rangle
+\langle \mathcal{I}^{-1}(a_{2}),\; (\fl_{A}+\fr_{A})(a_{1})(a_{3})\rangle\\
&=\langle \mathcal{I}^{-1}(\fl_{A}(a_{1})(a_{2}))-(\fl_{A}^{\ast}+\fr_{A}^{\ast})(a_{1})
(\mathcal{I}^{-1}(a_{2})),\; a_{3}\rangle\\
&=0,
\end{align*}
which implies that $\mathfrak{B}_{\mathcal{I}}$ is a nondegenerate invariant bilinear
form on $(A, \diamond)$.

Finally, since $(r^{\natural})^{\ast}=-r^{\sharp}$ and
$\mathcal{I}=r^{\sharp}-r^{\natural}$, we have
\begin{align*}
&\;\mathfrak{B}_{\mathcal{I}}(a_{1},\; P(a_{2}))+\mathfrak{B}_{\mathcal{I}}(P(a_{1}),\; a_{2})
+\lambda\mathfrak{B}_{\mathcal{I}}(a_{1}, a_{2})\\
=&\;\lambda\Big(\langle\mathcal{I}^{-1}(a_{1}),\; r^{\natural}(\mathcal{I}^{-1}(a_{2}))\rangle
+\langle\mathcal{I}^{-1}(r^{\natural}(\mathcal{I}^{-1}(a_{1}))),\; a_{2}\rangle
+\langle\mathcal{I}^{-1}(a_{1}),\; a_{2}\rangle\Big)\\
=&\;\lambda\Big(-\langle \mathcal{I}^{-1}(r^{\sharp}(\mathcal{I}^{-1}(a_{1}))),\; a_{2}\rangle
+\langle\mathcal{I}^{-1}(r^{\natural}(\mathcal{I}^{-1}(a_{1}))),\; a_{2}\rangle
+\langle\mathcal{I}^{-1}(a_{1}),\; a_{2}\rangle\Big)\\
=&\;\lambda\Big(\langle \mathcal{I}^{-1}((r^{\natural}-r^{\sharp})(\mathcal{I}^{-1}(a_{1})))
+\mathcal{I}^{-1}(a_{1}),\; a_{2}\rangle\Big)\\
=&\;0,
\end{align*}
which implies that Eq. \eqref{quaRB} holds. Therefore, $P$ is a Rota-Baxter operator of
weight $\lambda$ on the quadratic Novikov algebra $(A, \diamond, \mathfrak{B}_{\mathcal{I}})$.

Conversely, since $\mathfrak{B}(-,-)$ is symmetric, we have
$\mathcal{I}_{\mathfrak{B}}^{\ast}=\mathcal{I}_{\mathfrak{B}}$. By the fact that
$\mathfrak{B}(a_{1},\; P(a_{2}))+\mathfrak{B}(P(a_{1}),\; a_{2})
+\lambda \mathfrak{B}(a_{1},\; a_{2})=0$, we get
$\langle\mathcal{I}_{\mathfrak{B}}^{-1}(a_{1}),\; P(a_{2})\rangle
+\langle\mathcal{I}_{\mathfrak{B}}^{-1}(P(a_{1})),\; a_{2}\rangle
+\lambda\langle\mathcal{I}_{\mathfrak{B}}^{-1}(a_{1}),\; a_{2}\rangle=0$,
for any $a_{1}, a_{2}\in A$. That is,
$P^{\ast}\mathcal{I}_{\mathfrak{B}}^{-1}+\mathcal{I}_{\mathfrak{B}}^{-1}P
+\lambda \mathcal{I}_{\mathfrak{B}}^{-1}=0$, and
$\mathcal{I}_{\mathfrak{B}}P^{\ast}+P\mathcal{I}_{\mathfrak{B}}
+\lambda\mathcal{I}_{\mathfrak{B}}=0$. Therefore,
$$
r^{\natural}=-(r^{\sharp})^{\ast}
=-\mbox{$\frac{1}{\lambda}$}\mathcal{I}_{\mathfrak{B}}^{\ast}(P^{\ast}+\lambda\id)
=\mbox{$\frac{1}{\lambda}$}P\mathcal{I}_{\mathfrak{B}},
$$
and $\mathcal{I}_{\mathfrak{B}}=r^{\sharp}-r^{\natural}$. Define a multiplication
$\diamond_{r}$ on $A^{\ast}$ by
$$
\xi\diamond_{r}\eta:=(\fl_{A}^{\ast}+\fr_{A}^{\ast})(r^{\sharp}(\xi))(\eta)
-\fr_{A}^{\ast}(r^{\natural}(\eta))(\xi),
$$
for any $\xi, \eta\in A^{\ast}$. Next we prove that the following equation holds:
\begin{align}
\mbox{$\frac{1}{\lambda}$}\mathcal{I}_\mathfrak{B}(\xi\diamond_{r} \eta)
&=(\mbox{$\frac{1}{\lambda}$}\mathcal{I}_{\mathfrak{B}}(\xi))\diamond_{P}
(\mbox{$\frac{1}{\lambda}$}\mathcal{I}_{\mathfrak{B}}(\eta)),     \label{rhomo}
\end{align}
where $\diamond_{P}$ is the product in the descendent Novikov algebra $A_{P}$.
By the fact that $\mathfrak{B}(a_{1}\diamond a_{2},\; a_{3})
+\mathfrak{B}(a_{2},\; a_{1}\diamond a_{3}+a_{3}\diamond a_{1})=0$, we get
$\langle\mathcal{I}^{-1}_{\mathfrak{B}}(\fl_{A}(a_{1})(a_{2})),\; a_{3}\rangle
-\langle(\fl_{A}^{\ast}+\fr_{A}^{\ast})(a_{1})(\mathcal{I}^{-1}_{\mathfrak{B}}(a_{2})),\;
a_{3}\rangle=0$, for any $a_{1}, a_{2}, a_{3}\in A$. That is,
$\mathcal{I}_{\mathfrak{B}}\fl_{A}^{\ast}(a_{1})=(\fl_{A}+\fr_{A})(a_{1})
\mathcal{I}_{\mathfrak{B}}$. Therefore, by Lemma \ref{lem:syminva},
the symmetric part $\mathfrak{s}(r)$ of $r$ is invariant.

Moreover, note that
\begin{align*}
\mathcal{I}_{\mathfrak{B}}(\xi\diamond_{r}\eta)
=&\;\mathcal{I}_{\mathfrak{B}}\big((\fl_{A}^{\ast}+\fr_{A}^{\ast})(r^{\sharp}(\xi))(\eta)
-\fr_{A}^{\ast}(r^{\natural}(\eta))(\xi)\big)\\
=&\;\fl_{A}(r^{\sharp}(\xi))(\mathcal{I}_{\mathfrak{B}}(\eta))
-\fl_{A}(r^{\natural}(\eta))(\mathcal{I}_{\mathfrak{B}}(\xi))
+\mathcal{I}_{\mathfrak{B}}(\fl_{A}^{\ast}(r^{\natural}(\eta))(\xi))\\
=&\;r^{\sharp}(\xi)\diamond(r^{\sharp}-r^{\natural})(\eta)
-r^{\natural}(\eta)\diamond(r^{\sharp}-r^{\natural})(\xi)
+r^{\natural}(\eta)\diamond(r^{\sharp}-r^{\natural})(\xi)
+(r^{\sharp}-r^{\natural})(\xi)\diamond r^{\natural}(\eta)\\
=&\;r^{\sharp}(\xi)\diamond r^{\sharp}(\eta)-r^{\natural}(\xi)\diamond r^{\natural}(\eta),
\end{align*}
and
\begin{align*}
\mathcal{I}_{\mathfrak{B}}(\xi)\diamond_{P}(\mathcal{I}_\mathfrak{B}(\eta))
=&\; P\mathcal{I}_{\mathfrak{B}}(\xi)\diamond\mathcal{I}_{\mathfrak{B}}(\eta)
+\mathcal{I}_{\mathfrak{B}}(\xi)\diamond P(\mathcal{I}_{\mathfrak{B}}(\eta))
+\lambda\mathcal{I}_{\mathfrak{B}}(\xi)\diamond \mathcal{I}_{\mathfrak{B}}(\eta)\\
=&\;\lambda r^{\natural}(\xi)\diamond(r^{\sharp}-r^{\natural})(\eta)
+\lambda(r^{\sharp}-r^{\natural})(\xi)\diamond r^{\natural}(\eta)
+\lambda(r^{\sharp}-r^{\natural})(\xi)\diamond(r^{\sharp}-r^{\natural})(\eta)\\
=&\;\lambda r^{\sharp}(\xi)\diamond r^{\sharp}(\eta)
-\lambda r^{\natural}(\xi)\diamond r^{\natural}(\eta),
\end{align*}
for any $\xi, \eta\in A^{\ast}$, we get that Eq. \eqref{rhomo} holds.
Therefore, $\frac{1}{\lambda}\mathcal{I}_{\mathfrak{B}}$ is a Novikov algebra isomorphism
from $(A^{\ast}, \diamond_{r})$ to $(A, \diamond_{P})$.

Finally, by the fact that $P+\lambda id,\; P:(A, \diamond_{P})\rightarrow (A, \diamond)$
are both Novikov algebra homomorphisms, we deduce that
\begin{align*}
r^{\sharp}=\mbox{$\frac{1}{\lambda}$}(P+\lambda\id)\mathcal{I}_{\mathfrak{B}},\quad
r^{\natural}=\mbox{$\frac{1}{\lambda}$}P\mathcal{I}_{\mathfrak{B}}:\qquad\quad
(A^{\ast}, \diamond_{r})\longrightarrow(A, \diamond),
\end{align*}
are both Novikov algebra homomorphisms. Therefore, by Proposition \ref{pro:solu-equiv},
$r$ is a solution of the NYBE in $(A, \diamond)$ and $(A, \diamond, \delta_{r})$ is a
quasi-triangular Novikov bialgebra. Since $\mathcal{I}_{\mathfrak{B}}=r^{\sharp}
-r^{\natural}$ is a Novikov algebra isomorphism, the Novikov bialgebra $(A, \diamond,
\delta_{r})$ is factorizable.
\end{proof}

\begin{cor}\label{cor:corre}
Let $(A, \diamond)$ be a Novikov algebra and $r\in A\otimes A$. Suppose the Novikov
bialgebra $(A, \diamond, \delta_{r})$ is factorizable and
$\mathcal{I}=r^{\sharp}-r^{\natural}$. Then $-\lambda\id-P$ is also a Rota-Baxter
operator of weight $\lambda$ on quadratic Novikov algebra $(A, \diamond,
\mathfrak{B}_{\mathcal{I}})$, where $\mathfrak{B}_{\mathcal{I}}$ and $P$ are defined
in Theorem \ref{thm:quafact}.
\end{cor}

\begin{proof}
It can be obtained by direct calculation.
\end{proof}

\begin{cor}\label{cor:iso}
Let $(A, \diamond)$ be a Novikov algebra and $r\in A\otimes A$. Suppose the Novikov
bialgebra $(A, \diamond, \delta_{r})$ is factorizable, $\mathcal{I}=r^{\sharp}-r^{\natural}$
and $P=\lambda r^{\natural}\mathcal{I}^{-1}: A\rightarrow A$, where $0\neq\lambda\in\Bbbk$.
Then $(A, \diamond_{P}, \delta_{\mathcal{I}})$ is a Novikov bialgebra, where
$$
\delta_{\mathcal{I}}^{\ast}(\xi, \eta)=-\mbox{$\frac{1}{\lambda}$}
\mathcal{I}^{-1}(\mathcal{I}(\xi)\diamond\mathcal{I}(\eta)),
$$
for any $\xi, \eta\in A^{\ast}$. Moreover, $\frac{1}{\lambda}\mathcal{I}: A^{\ast}
\rightarrow A$ gives a Novikov bialgebra isomorphism from $(A^{\ast}, \diamond_{r},
\delta_{A^{\ast}})$ to $(A, \diamond_{P}, \delta_{\mathcal{I}})$, where
$\delta_{A^{\ast}}^{\ast}(a_{1}, a_{2})=a_{1}\diamond a_{2}$, for any $a_{1}, a_{2}\in A$.
\end{cor}

\begin{proof}
First we show that $\frac{1}{\lambda}\mathcal{I}$ is a Novikov algebra isomorphism
from $(A^{\ast}, \diamond_{r})$ to $(A, \diamond_{P})$. In fact, for any $\xi, \eta\in
A^{\ast}$, taking $a_{1}=\mathcal{I}(\xi)\in A$ and $a_{2}=\mathcal{I}(\eta)\in A$,
by Eq. \eqref{I1}, we have
$$
\mbox{$\frac{1}{\lambda}$}\mathcal{I}(\xi\diamond_{r}\eta)
=\mbox{$\frac{1}{\lambda^{2}}$}\Big(P(\mathcal{I}(\xi))\diamond\mathcal{I}(\eta)
+\mathcal{I}(\xi)\cdot P(\mathcal{I}(\eta))
+\lambda\mathcal{I}(\xi)\diamond\mathcal{I}(\eta)\Big)
=\mbox{$\frac{1}{\lambda}$}\mathcal{I}(\xi)\diamond_{P}
\mbox{$\frac{1}{\lambda}$}\mathcal{I}(\eta).
$$
Thus, $\frac{1}{\lambda}\mathcal{I}$ is a Novikov algebra isomorphism.
Moreover, since
$(\frac{1}{\lambda}\mathcal{I})^{\ast}=(\frac{1}{\lambda}\mathcal{I})$, we have
$$
(\mbox{$\frac{1}{\lambda}$}\mathcal{I})^{\ast}(\xi\diamond_\mathcal{I}\eta)
=\mbox{$\frac{1}{\lambda}$}\mathcal{I}(\xi)\diamond\mbox{$\frac{1}{\lambda}$}
\mathcal{I}(\eta)=(\mbox{$\frac{1}{\lambda}$}\mathcal{I})^{\ast}(\xi)\diamond
(\mbox{$\frac{1}{\lambda}$}\mathcal{I})^{\ast}(\eta),
$$
which means that the map$(\frac{1}{\lambda}\mathcal{I})^{\ast}=\frac{1}{\lambda}\mathcal{I}:
(A^{\ast}, \diamond_{\mathcal{I}})\rightarrow(A, \diamond)$ is also a Novikov algebra
isomorphism. Therefore, $(A, \diamond_{P}, \delta_{\mathcal{I}})$ is a Novikov bialgebra,
and one can check that $\frac{1}{\lambda}\mathcal{I}: (A^{\ast}, \diamond_{r},
\delta_{A^{\ast}})\rightarrow(A, \diamond_{P}, \delta_{\mathcal{I}})$
is a Novikov bialgebra isomorphism.
\end{proof}

\bigskip
\section{Quasi-triangular Novikov bialgebras from differential infinitesimal bialgebras}
\label{sec:factdiff}
In this section, we show that a solution $r$ of the NYBE in some special admissible
differential algebra such that $\mathfrak{s}(r)$ is $\fu_{A}$-invariant is a solution of
the NYBE in the induced Novikov algebra such that the symmetric part is invariant.
Therefore, under certain conditions, we show that the quasi-triangular
Novikov bialgebras induced by a quasi-triangular (resp. triangular, factorizable) differential
infinitesimal bialgebra is also quasi-triangular (resp. triangular, factorizable).
Let $(A, \cdot, \partial)$ be a differential algebra and $\theta: A\rightarrow A$
be a linear map. If for any $a_{1}, a_{2}\in A$,
$$
\theta(a_{1}a_{2})=\theta(a_{1})a_{2}-a_{1}\partial(a_{2}),
$$
then we say that $\theta$ is {\it admissible to $(A, \cdot, \partial)$},
where $a_{1}a_{2}:=a_{1}\cdot a_{2}$ for simply. We also say that the quadruple
$(A, \cdot, \partial, \theta)$ is an {\it admissible differential algebra}.

\begin{pro}[\cite{HBG1}]\label{pro:ind-nov}
Let $(A, \cdot, \partial, \theta)$ be an admissible differential algebra. Define a
binary operation $\diamond_{\q}$ on $A$ by
\begin{align}
a_{1}\diamond_{\q}a_{2}:=a_{1}(\partial+\q\theta)(a_{2}),  \label{ind-nov}
\end{align}
for any $a_{1}, a_{2}\in A$, where $\q\in\Bbbk$. Then $(A, \diamond_{\q})$ is a
Novikov algebra for all $\q\in\Bbbk$. We call $(A, \diamond_{\q})$ the
{\rm Novikov algebra induced from $(A, \cdot, \partial, \theta)$}.
\end{pro}

Recall that a {\it coassociative coalgebra} is a vector space $A$ with a linear map
$\Delta: A\rightarrow A\otimes A$ satisfying $(\Delta\otimes\id)\Delta
=(\id\otimes\Delta)\Delta$. A coassociative coalgebra $(A, \delta)$ is called
{\it cocommutative} if $\Delta=\tau\Delta$. For a coassociative coalgebra
$(A, \Delta)$, a linear map $\theta: A\rightarrow A$ is called a {\it coderivation}
on $(A, \Delta)$ if $\Delta\theta=(\id\otimes\theta+\theta\otimes\id)\Delta$.
A {\it codifferential coalgebra} is a triple $(A, \Delta, \theta)$ where
$(A, \Delta)$ is a cocommutative coassociative coalgebra and $\theta$ is a
coderivation. Obviously, $(A, \Delta, \theta)$ is a codifferential
coalgebra if and only if $(A^{\ast}, \Delta^{\ast}, \theta^{\ast})$ is a differential algebra.
Let $(A, \Delta, \theta)$ be a codifferential coalgebra. If $\partial: A\rightarrow A$
is a linear map satisfying
$$
(\partial\otimes\id-\id\otimes\theta)\Delta=\Delta\partial,
$$
then we call $(A, \Delta, \theta, \partial)$ an {\it admissible codifferential
coalgebra}. One can check that $(A, \Delta, \theta, \partial)$ is an admissible
codifferential coalgebra if and only if $(A, \Delta^{\ast}, \theta^{\ast}, \partial^{\ast})$
is an admissible differential algebra.

\begin{pro}[\cite{HBG1}]\label{pro:ind-conov}
Let $(A, \Delta, \theta, \partial)$ be an admissible codifferential coalgebra.
Define a coproduct $\delta_{\q}$ on $A$ by
\begin{align}
\delta_{\q}=(\id\otimes(\theta+\q\partial))\Delta,  \label{ind-conov}
\end{align}
where $\q\in\Bbbk$. Then $(A, \delta_{\q})$ is a Novikov coalgebra for all
$\q\in\Bbbk$. We call $(A, \Delta_{\q})$ the {\rm Novikov coalgebra induced from
$(A, \Delta, \theta, \partial)$}.
\end{pro}

Let $(A, \cdot)$ be a commutative associative algebra and $(A, \Delta)$ be a
commutative coassociative coalgebra. If for any $a_{1}, a_{2}\in A$,
\begin{align}
\Delta(a_{1}a_{2})=(\id\otimes\fu_{A}(a_{1}))(\Delta(a_{2}))
+(\fu(a_{2})\otimes\id)(\Delta(a_{1})),   \label{comm-bia}
\end{align}
then $(A, \cdot, \Delta)$ is called an {\it infinitesimal bialgebra}.

\begin{defi}[\cite{LLB}]\label{def:diff-bia}
A quintuple $(A, \cdot, \Delta, \partial, \theta)$ is called a {\rm differential
infinitesimal bialgebra} if
\begin{enumerate}
\item $(A, \cdot, \Delta)$ is an infinitesimal bialgebra,
\item  $(A,\cdot, \partial, \theta)$ is an admissible differential algebra, and
\item  $(A, \Delta, \theta, \partial)$ is an admissible codifferential coalgebra.
\end{enumerate}
\end{defi}

\begin{thm}[\cite{HBG1}]\label{thm:ind-novbia}
Let $(A, \cdot, \Delta, \partial, \theta)$ be a differential infinitesimal bialgebra.
Let $\q\in\Bbbk$ be given. Define a binary operation $\diamond_{\q}$ and a
coproduct $\delta_{\q}$ on $A$ by Eqs. \eqref{ind-nov} and \eqref{ind-conov} respectively.
\begin{enumerate}
\item When $\q=-\frac{1}{2}$, then $(A, \diamond_{\q}, \delta_{\q})$ is a
     Novikov bialgebra.
\item For other $\q\in\Bbbk$, if $(\id\otimes\theta)\Delta=-(\id\otimes\partial)\Delta$,
     and $a_{1}\theta(a_{2})=-a_{1}\partial(a_{2})$ for any $a_{1}, a_{2}\in A$,
     then $(A, \diamond_{\q}, \delta_{\q})$ is a Novikov bialgebra. In particular,
     if $\theta=-\partial$, then $(A, \diamond_{\q}, \delta_{\q})$ is a Novikov bialgebra.
\end{enumerate}
\end{thm}

The Novikov bialgebra $(A, \diamond_{\q}, \delta_{\q})$ in Theorem \ref{thm:ind-novbia}
is called the {\it Novikov bialgebra induced by the differential infinitesimal bialgebra
$(A, \cdot, \Delta, \partial, \theta)$}.
Recall  a {\it derivation} on an infinitesimal bialgebra $(A, \cdot, \Delta)$ is defined
to be a linear operator on $A$ that is both a derivation on $(A, \cdot)$ and a
coderivation on $(A, \Delta)$. Let $(A, \cdot, \Delta, \partial, \theta)$ be a
differential infinitesimal bialgebra. If both $\partial$ and $\theta$ are derivations
on the infinitesimal bialgebra $(A, \cdot, \Delta)$, then $(A, \diamond_{\q}, \delta_{\q})$
is a Novikov bialgebra for all $\q\in\Bbbk$, where $\diamond_{\q}$ and $\delta_{\q}$
are defined by Eqs. \eqref{ind-nov} and \eqref{ind-conov} respectively.

\begin{ex}\label{ex:ind-Novbia}
Let $(A=\Bbbk\{e_{1}, e_{2}\}, \cdot)$ be the commutative associative algebra whose
multiplication is given by
$$
e_{1}e_{1}=e_{1},\qquad e_{1}e_{2}=e_{2}e_{1}=e_{2},\qquad e_{2}e_{2}=0.
$$
Let $\partial, \theta: A\rightarrow A$ be the linear maps given by
$$
\partial(e_{1})=0,\quad\partial(e_{2})=e_{2},\qquad\quad
\theta(e_{1})=e_{1},\quad \theta(e_{2})=0.
$$
Then $(A, \cdot, \partial, \theta)$ is an admissible differential algebra.
Define $\Delta: A\rightarrow A\otimes A$ by $\Delta(e_{1})=0$ and
$\Delta(e_{2})=e_{2}\otimes e_{2}$. Then, one can check that $(A, \cdot, \Delta,
\partial, \theta)$ is a differential infinitesimal bialgebra.
By Theorem \ref{thm:ind-novbia} for $\q=-\frac{1}{2}$, we get a Novikov bialgebra
$(A, \diamond, \delta)$, where
\begin{align*}
&e_{1}\diamond e_{1}=-\mbox{$\frac{1}{2}$}e_{1},\qquad e_{1}\diamond e_{2}=e_{2},\qquad e_{2}\diamond e_{1}=-\mbox{$\frac{1}{2}$}e_{2},\qquad e_{2}\diamond e_{2}=0,\\
&\qquad\qquad\delta(e_{1})=0,\qquad\qquad\qquad\qquad
\delta(e_{2})=-\mbox{$\frac{1}{2}$}e_{2}\otimes e_{2}.
\end{align*}
\end{ex}

In this section, we consider some special Novikov bialgebras induced by differential
infinitesimal bialgebras. Recall that an infinitesimal bialgebra $(A, \cdot, \Delta)$
is called {\it coboundary} if there exists a $r\in A\otimes A$ such that
$\Delta=\Delta_{r}$, where
\begin{align}
\Delta_{r}(a)=(\id\otimes\fu_{A}(a)-\fu_{A}(a)\otimes\id)(r), \label{comm-cobo}
\end{align}
for any $a\in A$. Let $(A, \cdot)$ be a commutative associative algebra. Recall that
an element $r=\sum_{i}x_{i}\otimes y_{i}\in A\otimes A$ is said to be
{\it $\fu_{A}$-invariant} if $(\id\otimes\fu_{A}(a)-\fu_{A}(a)\otimes\id)(r)=0$. The equation
$$
\mathbf{A}_{r}:=r_{13}r_{12}+r_{13}r_{23}-r_{12}r_{23}=0
$$
is called the {\it associative Yang-Baxter equation (or, AYBE)} in $(A, \cdot)$,
where $r_{13}r_{12}=\sum_{i,j}x_{i}x_{j}\otimes y_{j}\otimes y_{i}$,
$r_{13}r_{23}=\sum_{i,j}x_{i}\otimes x_{j}\otimes y_{i}y_{j}$ and
$r_{12}r_{23}=\sum_{i,j}x_{i}\otimes y_{i}x_{j}\otimes y_{j}$.
Note that $\Delta_{r}$ defined by Eq. \eqref{comm-cobo} satisfies Eq. \eqref{comm-bia},
and by \cite{Bai}, $(A, \Delta_{r})$ is a coassociative coalgebra if and only if
$$
(\fu_{A}(a)\otimes\id\otimes\id-\id\otimes\id\otimes\fu_{A}(a))(\mathbf{A}_{r})=0,
$$
for any $a\in A$. Thus, we get the following lemma.

\begin{lem}\label{lem:comm-bia}
Let $(A, \cdot)$ be a commutative associative algebra, $r\in A\otimes A$ and
$\Delta_{r}: A\rightarrow A\otimes A$ defined by Eq. \eqref{comm-cobo}.
\begin{enumerate}
\item If $r$ is a solution of the AYBE in $(A, \cdot)$ and the symmetric part
     $\mathfrak{s}(r)$ of $r$ is $\fu_{A}$-invariant, then $(A, \cdot, \Delta_{r})$ is an
     infinitesimal bialgebra, which is called a {\rm quasi-triangular infinitesimal
     bialgebra} associated with $r$.
\item If $r$ is a skew-symmetric solution of the AYBE in $(A, \cdot)$ then $(A, \cdot,
     \Delta_{r})$ is an infinitesimal bialgebra, which is called a {\rm triangular
     infinitesimal bialgebra} associated with $r$.
\end{enumerate}
\end{lem}

Let $(A, \cdot, \partial, \theta)$ be an admissible differential algebra and
$r\in A\otimes A$. If $r$ is a solution of the AYBE in $(A, \cdot)$ and satisfies
$$
(\partial\otimes\id-\id\otimes\theta)(r)=0,\qquad\qquad
(\id\otimes\partial-\theta\otimes\id)(r)=0,
$$
then $r$ is called a solution of the {\it admissible AYBE} in $(A, \cdot, \partial, \theta)$.
One can check that $\tau(r)$ is a solution of the {\it admissible AYBE} in
$(A, \cdot, \partial, \theta)$ if and only if $r\in A\otimes A$ is a solution of
the {\it admissible AYBE} in $(A, \cdot, \partial, \theta)$.

\begin{lem}[\cite{LLB}]\label{lem:adim-co}
Let $(A, \cdot, \partial, \theta)$ be an admissible differential algebra, $r\in A\otimes A$
and $\Delta_{r}: A\rightarrow A\otimes A$ defined by Eq. \eqref{comm-cobo}.
Then $(A, \Delta_{r}, \theta, \partial)$ be an admissible codifferential coalgebra
if for any $a\in A$,
\begin{align*}
&(\id\otimes\fu_{A}(a))((\id\otimes\partial-\theta\otimes\id)(r))
+(\fu_{A}(a)\otimes\id)((\id\otimes\theta-\partial\otimes\id)(r))=0,\\
&\qquad\qquad(\id\otimes\fu_{A}(a)-\fu_{A}(a)\otimes\id)((\partial\otimes\id
-\id\otimes\theta)(r))=0,\\
&\qquad\qquad(\id\otimes\fu_{A}(a)-\fu_{A}(a)\otimes\id)((\id\otimes\theta
-\partial\otimes\id)(r))=0.
\end{align*}
\end{lem}

By Lemmas \ref{lem:comm-bia} and \ref{lem:adim-co}, we get the following proposition.

\begin{pro}\label{pro:diff-bia}
Let $(A, \cdot, \partial, \theta)$ be an admissible differential algebra,
$r\in A\otimes A$ and $\Delta_{r}: A\rightarrow A\otimes A$ defined by Eq. \eqref{comm-cobo}.
\begin{enumerate}
\item If $r$ is a solution of the admissible AYBE in $(A, \cdot, \partial, \theta)$ and
     $\mathfrak{s}(r)$ is $\fu_{A}$-invariant, then $(A, \cdot, \Delta_{r}$, $\partial,
     \theta)$ is a differential infinitesimal bialgebra, which is called a {\rm
     quasi-triangular differential infinitesimal bialgebra} associated with $r$.
\item If $r$ is a skew-symmetric solution of the admissible AYBE in $(A, \cdot,
     \partial, \theta)$ then $(A, \cdot, \Delta_{r}, \partial, \theta)$ is a differential
     infinitesimal bialgebra, which is called a {\rm triangular differential infinitesimal
     bialgebra} associated with $r$.
\end{enumerate}
\end{pro}

Next, we give a construction of solutions of NYBE in the induced Novikov algebra
from solutions of admissible AYBE in an admissible differential algebra.

\begin{pro}\label{pro:ind-YBE}
Let $(A, \cdot, \partial, \theta)$ be an admissible differential algebra, $r=\sum_{i}x_{i}
\otimes y_{i}\in A\otimes A$ and $(A, \diamond_{\q})$, $\q\in\Bbbk$, be the induced
Novikov algebra. Whenever $\q=-\frac{1}{2}$, or $\theta$ is a derivation on $(A, \cdot)$,
we get $(1)$ $r$ is a solution of the NYBE in $(A, \diamond_{\q})$
if $r$ is a solution of the admissible AYBE in $(A, \cdot, \partial, \theta)$;
$(2)$ $\mathfrak{s}(r)$ is invariant if $r$ is a solution of the admissible AYBE in $(A, \cdot,
\partial, \theta)$ and $\mathfrak{s}(r)$ is $\fu_{A}$-invariant.
\end{pro}

\begin{proof}
$(1)$ If $r$ is a solution of the admissible AYBE in $(A, \cdot, \partial, \theta)$,
we get $\sum_{i}\partial(x_{i})\otimes y_{i}=\sum_{i}x_{i}\otimes\theta(y_{i})$,
$\sum_{i}\theta(x_{i})\otimes y_{i}=\sum_{i}x_{i}\otimes\partial(y_{i})$, and
$\mathbf{A}_{r}=0$. Moreover, since $\theta$ is admissible to $(A, \cdot, \partial)$ and
$(A, \cdot)$ is commutative, for any $a_{1}, a_{2}\in A$, we have
$a_{1}(\partial+\theta)(a_{2})=(\partial+\theta)(a_{1})a_{2}$. Thus, we get
\begin{align*}
\mathbf{N}_{r}
&=\sum_{i,j}\Big(x_{i}\otimes x_{j}\otimes y_{i}(\partial+\q\theta)(y_{j})
+x_{i}\otimes y_{i}(\partial+\q\theta)(x_{j})\otimes y_j\\[-5mm]
&\qquad\quad+x_{i}\otimes x_{j}(\partial+\q\theta)(y_i)\otimes y_j
+x_{i}(\partial+\q\theta)(x_{j})\otimes y_{j}\otimes y_{i}\Big)\\
&=\sum_{i,j}\Big(x_{i}\otimes(\theta+\q\partial)(x_{j})\otimes y_{i}y_{j}
+x_{i}\otimes y_{i}(\partial+\q\theta)(x_{j})\otimes y_{j}\\[-5mm]
&\qquad\quad+x_{i}\otimes x_{j}(\partial+\q\theta)(y_{i})\otimes y_{j}
+x_{i}x_{j}\otimes(\theta+\q\partial)(y_{j})\otimes y_{i}\Big)\\
&=\sum_{i,j}x_{i}\otimes\Big((\theta+\q\partial)(x_{j}y_{i})+y_{i}(\partial+\q\theta)(x_{j})
+x_{j}(\partial+\q\theta)(y_{i})\Big)\otimes y_{j}\\
&=\sum_{i,j}x_{i}\otimes(1+2\q)(\partial+\theta)(x_{j})y_{i}\otimes y_{j}.
\end{align*}
Note that $(\partial+\theta)(x_{j})y_{i}=0$ for any $x_{j}, y_{i}\in A$ if $\theta$
is a derivation, we get that if $\q=-\frac{1}{2}$, or $\theta$ is a derivation,
$\mathbf{N}_{r}=0$, $r$ is a solution of the NYBE in $(A, \diamond_{\q})$.

$(2)$ If $r$ is a solution of the admissible AYBE in $(A, \cdot, \partial, \theta)$ and
$\mathfrak{s}(r)$ is $\fu_{A}$-invariant, then $\tau(r)$ is also a solution of the admissible
AYBE in $(A, \cdot, \partial, \theta)$. If denote $r+\tau(r)=\sum_{i}\bar{x}_{i}
\otimes\bar{y}_{i}$, then we also have $\sum_{i}\partial(\bar{x}_{i})\otimes\bar{y}_{i}
=\sum_{i}\bar{x}_{i}\otimes\theta(\bar{y}_{i})$ and $\sum_{i}\theta(\bar{x}_{i})
\otimes\bar{y}_{i}=\sum_{i}\bar{x}_{i}\otimes\partial(\bar{y}_{i})$.
Thus, for any $a\in A$, we get
\begin{align*}
&\;\big(\fl_{A}(a)\otimes\id+\id\otimes(\fl_{A}+\fr_{A})(a)\big)(r+\tau(r))\\
=&\;\sum_{i}(a\diamond_{\q}\bar{x}_{i})\otimes\bar{y}_{i}
+\bar{x}_{i}\otimes(a\diamond_{\q}\bar{y}_{i})
+\bar{x}_{i}\otimes(\bar{y}_{i}\diamond_{\q}a)\\[-2mm]
=&\;\sum_{i}(a(\partial+\q\theta)(\bar{x}_{i}))\otimes\bar{y}_{i}
+\bar{x}_{i}\otimes(a(\partial+\q\theta)(\bar{y}_{i}))
+\bar{x}_{i}\otimes(\bar{y}_{i}(\partial+\q\theta)(a))\\[-2mm]
=&\;\sum_{i}\bar{x}_{i}\otimes\Big(\theta(a\bar{y}_{i})+\q\partial(a\bar{y}_{i})
+a\partial(\bar{y}_{i})+\q a\theta(\bar{y}_{i})+\partial(a)\bar{y}_{i}
+\q\theta(a)\bar{y}_{i}\Big)\\[-2mm]
=&\;\sum_{i}\bar{x}_{i}\otimes(1+2\q)a(\partial+\theta)(\bar{y}_{i}),
\end{align*}
since $\mathfrak{s}(r)$ is $\fu_{A}$-invariant, i.e., $\sum_{i}a\bar{x}_{i}\otimes\bar{y}_{i}
=\sum_{i}\bar{x}_{i}\otimes a\bar{y}_{i}$ for any $a\in A$. Note that
$a(\partial+\theta)(\bar{y}_{i})=0$ for any $a, y_{i}\in A$ if $\theta$ is a derivation,
we get that if $\q=-\frac{1}{2}$, or $\theta$ is a derivation, $\mathfrak{s}(r)$ is invariant
for the Novikov algebra $(A, \diamond_{\q})$.
\end{proof}

By $(1)$ in Proposition \ref{pro:ind-YBE}, we obtain the following corollary directly.

\begin{cor}[\cite{HBG1}]\label{cor:ind-s-YBE}
Let $(A, \cdot, \partial, \theta)$ be an admissible differential algebra, $r\in A\otimes A$
and $(A, \diamond_{\q})$, $\q\in\Bbbk$, be the induced Novikov algebra. Whenever
$\q=-\frac{1}{2}$, or $\theta$ is a derivation on $(A, \cdot)$,
we get $r$ is a skew-symmetric solution of the NYBE in $(A, \diamond_{\q})$
if $r$ is a skew-symmetric solution of the admissible AYBE in $(A, \cdot, \partial, \theta)$.
\end{cor}

Now, we introduce the definition of factorizable differential infinitesimal bialgebra.
Let $(A, \cdot)$ be a commutative associative algebra. For any $r\in A\otimes A$,
we can also define linear map $r^{\sharp}: A^{\ast}\rightarrow A$ by
$\langle r^{\sharp}(\xi_{1}),\; \xi_{2}\rangle=\langle\xi_{1}\otimes\xi_{2},\; r\rangle$,
linear map $r^{\natural}: A^{\ast}\rightarrow A$ by
$\langle\xi_{1},\; r^{\natural}(\xi_{2})\rangle=-\langle\xi_{1}\otimes\xi_{2},\; r\rangle$,
for any $\xi_{1}, \xi_{2}\in A^{\ast}$, and linear map $\mathcal{I}=r^{\sharp}-r^{\natural}:
A^{\ast}\rightarrow A$.

\begin{defi}\label{def:fact-diffbia}
Let $(A, \cdot, \partial, \theta)$ be an admissible differential algebra, $r\in A\otimes A$,
and $(A, \cdot, \Delta_{r}, \partial, \theta)$ be a quasi-triangular differential
infinitesimal bialgebra associated with $r$. If $\mathcal{I}=r^{\sharp}-r^{\natural}:
A^{\ast}\rightarrow A$ is a bijection and $\mathcal{I}\theta^{\ast}=\partial\mathcal{I}$,
then $(A, \cdot, \Delta_{r}, \partial, \theta)$ is called a {\rm factorizable
differential infinitesimal bialgebra}.
\end{defi}

Similar to Proposition \ref{pro:fact}, we have

\begin{pro}\label{pro:fact-diff}
Let $(A, \cdot, \partial, \theta)$ be an admissible differential algebra and $r\in A\otimes A$.
Assume that the differential infinitesimal bialgebra $(A, \cdot, \Delta_{r}, \partial, \theta)$
is factorizable. Then any $a\in A$ has an unique decomposition $a=a_{+}+a_{-}$,
where $a_{+}\in\Img(r^{\sharp})$ and $a_{-}\in\Img(r^{\natural})$.
\end{pro}

Let $(A, \cdot, \Delta, \partial, \theta)$ be a differential infinitesimal bialgebra.
Denote by $\bar{\Delta}: A^{\ast}\rightarrow A^{\ast}\otimes A^{\ast}$ the dual
comultiplication of multiplication $``\cdot"$, and $\circ: A^{\ast}\otimes A^{\ast}
\rightarrow A^{\ast}$ the dual multiplication of comultiplication $``\Delta"$.
Then by Theorem 4.5 in \cite{LLB}, we get $(A^{\ast}, \circ, -\bar{\Delta},
\theta^{\ast}, \partial^{\ast})$ is also a differential infinitesimal bialgebra.
Moreover, let $\{e_{1}, e_{2}, \cdots, e_{n}\}$ be a basis of $A$, $\{f_{1}, f_{2}, \cdots,
f_{n}\}$ be its dual basis in $A^{\ast}$ and $\tilde{r}=\sum_{i=1}^{n}e_{i}\otimes f_{i}
\in A\otimes A^{\ast}\subset(A\oplus A^{\ast})\otimes(A\oplus A^{\ast})$.
Then, we get $(A\oplus A^{\ast}, \odot, \partial\oplus\theta^{\ast},
\theta\oplus\partial^{\ast})$ is a differential infinitesimal algebra, where
$$
(a_{1}, \xi_{1})\odot(a_{2}, \xi_{2})=\big(a_{1}a_{2}+\fu_{A^{\ast}}^{\ast}(\xi_{1})(a_{2})
+\fu_{A^{\ast}}^{\ast}(\xi_{2})(a_{1}),\ \ \xi_{1}\circ \xi_{2}
+\fu_{A}^{\ast}(a_{1})(\xi_{2})+\fu_{A}^{\ast}(a_{2})(\xi_{1})\big),
$$
for any $a_{1}, a_{2}\in A$, $\xi_{1}, \xi_{2}\in A^{\ast}$.
For the details see \cite{LLB}. Similar to Proposition \ref{pro:doufact}, we can show
that this type of differential infinitesimal bialgebra is always factorizable.

\begin{pro}\label{pro:doufactdiff}
Let $(A, \cdot, \Delta, \partial, \theta)$ be a differential infinitesimal bialgebra.
Using the notation above, we get that the differential infinitesimal bialgebra
$(A\oplus A^{\ast}, \odot, \Delta_{\tilde{r}}, \partial\oplus\theta^{\ast},
\theta\oplus\partial^{\ast})$ is a factorizable, which is called the {\rm
differential double of $(A, \cdot, \Delta, \partial, \theta)$}.
\end{pro}

Now, we consider the quasi-triangular, triangular and factorizable Novikov bialgebra
induced by the differential infinitesimal bialgebra. We get the main conclusion
of this section.

\begin{thm}\label{thm:indu-spbia}
Let $(A, \cdot, \partial, \theta)$ be an admissible differential algebra. Suppose
$(A, \cdot, \Delta, \partial, \theta)$ is a differential infinitesimal bialgebra
and $(A, \diamond_{\q}, \delta_{\q})$ is the Novikov bialgebra induced by
$(A, \cdot, \Delta, \partial, \theta)$. Then, if $\q=-\frac{1}{2}$, or $\theta$ is a
derivation on $(A, \cdot)$, we have
\begin{enumerate}\itemsep=0pt
\item $(A, \diamond_{\q}, \delta_{\q})$ is quasi-triangular if
     $(A, \cdot, \Delta, \partial, \theta)$ is quasi-triangular;
\item $(A, \diamond_{\q}, \delta_{\q})$ is triangular if
     $(A, \cdot, \Delta, \partial, \theta)$ is triangular;
\item $(A, \diamond_{\q}, \delta_{\q})$ is factorizable if
     $(A, \cdot, \Delta, \partial, \theta)$ is factorizable.
\end{enumerate}
\end{thm}

\begin{proof}
First, $(1)$ and $(2)$ follow from Proposition \ref{pro:ind-YBE} and Corollary
\ref{cor:ind-s-YBE}. Second, it is easy to see that the map
$\mathcal{I}=r^{\sharp}-r^{\natural}: A^{\ast}\rightarrow A$ for the induced
Novikov bialgebra $(A, \diamond_{\q}, \delta_{\q})$ is an isomorphism of vector
spaces if $\mathcal{I}$ is an isomorphism for the differential infinitesimal bialgebra
$(A, \cdot, \Delta, \partial, \theta)$. Thus we get $(3)$.
\end{proof}

Let $(A, \cdot, \partial, \theta)$ be an admissible differential algebra and $r\in A\otimes A$.
If $\q=-\frac{1}{2}$, or $\theta$ is a derivation on $(A, \cdot)$,
one can check the following diagram is commutative:
$$\xymatrix@C=3cm@R=0.6cm{
\txt{$r$ \\ {\tiny a solution of the admissible AYBE}\\ {\tiny in
$(A, \cdot, \partial, \theta)$ of type ($\star$)}} \ar[d]_{{\rm Pro.}~\ref{pro:ind-YBE}}
^{{\rm Cor.}~\ref{cor:ind-s-YBE}} \ar[r]^{{\rm Pro.}~\ref{pro:diff-bia}} &
\txt{$(A, \cdot, \Delta_{r}, \partial, \theta)$ \\ {\tiny a differential infinitesimal
bialgebra}} \ar[d]^{{\rm Thm.}~\ref{thm:ind-novbia}} \\
\txt{$r$ \\ {\tiny a solution of the NYBE}\\ {\tiny in $(A, \diamond_{\q})$ of type
($\star$)}} \ar[r]^{{\rm Pro.}~\ref{pro:tri-bialg},\quad {\rm Pro.}~\ref{pro:inv-bia}} &
\txt{$(A, \diamond_{\q}, \delta_{\q})=(A, \diamond_{\q}, \delta_{r})$ \\
{\tiny the induced Novikov bialgebra}}}
$$
where type ($\star$) means $(1)$ $\mathfrak{s}(r)$ is invariant, or $(2)$ $r$ is
skew-symmetric, or $(3)$ $\mathfrak{s}(r)$ is invariant and $\mathcal{I}=r^{\sharp}
-r^{\natural}$ is an isomorphism of vector spaces.

\begin{ex}\label{ex:ind-double}
Consider the 3-dimensional admissible differential algebra $(A=\Bbbk\{e_{1}, e_{2}, e_{3}\},
\cdot, \partial, \theta)$, where the multiplication $``\cdot"$ is given by $e_{1}e_{1}=e_{2}$,
$e_{1}e_{2}=e_{3}=e_{2}e_{1}$, the derivation $\partial$ is given by $\partial(e_{1})=e_{2}$,
$\partial(e_{2})=2e_{3}$, and $\theta=-\partial$. Let $r=e_{2}\otimes e_{3}
-e_{3}\otimes e_{2}$. Then, one can check that $r$ is a skew-symmetric solution
of the admissible AYBE in $(A, \cdot, \partial, \theta)$. Thus, by Proposition
\ref{pro:diff-bia}, we get a differential infinitesimal bialgebra $(A, \cdot,
\Delta_{r}, \partial, \theta)$, where the comultiplication $\Delta_{r}: A\rightarrow
A\otimes A$ is given by $\Delta_{r}(e_{1})=-2e_{3}\otimes e_{3}$, $\Delta_{r}(e_{2})=0
=\Delta_{r}(e_{3})$. Therefore, by Theorem \ref{thm:ind-novbia}, we get a Novikov
bialgebra $(A, \diamond_{\q}, \delta_{\q})$, where the comultiplication
$\delta_{\q}: A\rightarrow A\otimes A$ is trivial, i.e., $\delta_{\q}=0$.

On the other hand, by Proposition \ref{pro:ind-nov},
we get a Novikov algebra $(A, \diamond_{\q})$, where the binary operation $\diamond_{\q}$
is given by $e_{1}\diamond_{\q}e_{1}=(1-\q)e_{3}$, otherwise it is zero.
One can check that $r=e_{2}\otimes e_{3}-e_{3}\otimes e_{2}$ is a skew-symmetric solution
of the NYBE in $(A, \diamond_{\q})$. By Proposition \ref{pro:tri-bialg}, we get a
Novikov bialgebra $(A, \diamond_{\q}, \delta_{r})$. By direct calculation, we obtain
$\delta_{r}=0=\delta_{\q}$, and so that, the diagram above is commutative.
\end{ex}

\bigskip
\section{Quasi-triangular Novikov bialgebras to quasi-triangular Lie bialgebras}
\label{sec:factlie}

In \cite{LS}, Lang and Sheng have studied quasi-triangular Lie bialgebras and
factorizable Lie bialgebras. Recently, Hong, Bai and Guo provide a method for
constructing Lie bialgebras using Novikov bialgebras and quadratic right Novikov
algebras \cite{HBG}. In this section, we consider the quasi-triangular Lie bialgebras
from this construction. Let us first review some basic facts of Lie bialgebras.
Let $(\g, [-,-])$ be a Lie algebra, i.e., $[g_{1}, g_{2}]=-[g_{2}, g_{1}]$
and $[[g_{1}, g_{2}], g_{3}]+[[g_{2}, g_{3}], g_{1}]+[[g_{3}, g_{1}], g_{2}]=0$
for any $g_{1}, g_{2}, g_{3}\in\g$. A {\it Lie coalgebra} $(\g, \tilde{\Delta})$
is a linear vector $\g$ with a linear map $\tilde{\Delta}: \g\rightarrow\g\otimes\g$,
such that $\tau\tilde{\Delta}=-\tilde{\Delta}$ and $(\id\otimes\tilde{\Delta})
\tilde{\Delta}-(\tau\otimes\id)(\id\otimes\tilde{\Delta})\tilde{\Delta}
=(\tilde{\Delta}\otimes\id)\tilde{\Delta}$. One can check that $(\g, \tilde{\Delta})$
is a Lie coalgebra if and only if $(\g^{\ast}, \tilde{\Delta}^{\ast})$ is a Lie algebra.

A {\it Lie bialgebra} is a triple $(\g, [-,-], \tilde{\Delta})$ such that $(\g, [-,-])$ is a
Lie algebra, $(\g, \tilde{\Delta})$ is a Lie coalgebra, and the following
compatibility condition holds:
$$
\tilde{\Delta}([g_{1}, g_{2}])=(\ad_{\g}(g_{1})\otimes\id+\id\otimes\ad_{\g}(g_{1}))
(\tilde{\Delta}(g_{2}))-(\ad_{\g}(g_{2})\otimes\id+\id\otimes\ad_{\g}(g_{2}))
(\tilde{\Delta}(g_{1})),
$$
where $\ad_{\g}(g_{1})(g_{2})=[g_{1}, g_{2}]$, for all $g_{1}, g_{2}\in\g$.
A Lie bialgebra $(\g, [-,-], \tilde{\Delta})$ is called {\it coboundary} if there exists a
$r\in\g\otimes\g$ such that $\tilde{\Delta}=\tilde{\Delta}_{r}$, where
\begin{align}
\tilde{\Delta}_{r}(g)=(\id\otimes\ad_{\g}(g)+\ad_{\g}(g)\otimes\id)(r), \label{lie-cobo}
\end{align}
for any $g\in\g$. Let $(\g, [-,-])$ be a Lie algebra. Recall that an element
$r=\sum_{i}x_{i}\otimes y_{i}\in\g\otimes\g$ is said to be {\it $\ad_{\g}$-invariant} if
$(\id\otimes\ad_{\g}(g)+\ad_{\g}(g)\otimes\id)(r)=0$. The equation
$$
\mathbf{C}_{r}:=[r_{12}, r_{13}]+[r_{13}, r_{23}]+[r_{12}, r_{23}]=0
$$
is called the {\it classical Yang-Baxter equation (or, CYBE)} in $(\g, [-,-])$,
where $[r_{12}, r_{13}]=\sum_{i,j}[x_{i}, x_{j}]\otimes y_{i}\otimes y_{j}$,
$[r_{13}, r_{23}]=\sum_{i,j}x_{i}\otimes x_{j}\otimes[y_{i}, y_{j}]$ and
$[r_{12}, r_{23}]=\sum_{i,j}x_{i}\otimes[y_{i}, x_{j}]\otimes y_{j}$.

\begin{pro}[\cite{RS,LS}]\label{pro:lie-bia}
Let $(\g, [-,-])$ be a Lie algebra, $r\in\g\otimes\g$ and $\tilde{\Delta}_{r}:
\g\rightarrow\g\otimes\g$ defined by Eq. \eqref{lie-cobo}.
\begin{enumerate}
\item If $r$ is a solution of the CYBE in $(\g, [-,-])$ and $\mathfrak{s}(r)$ is
     $\ad_{\g}$-invariant, then $(\g, [-,-], \tilde{\Delta}_{r})$ is a Lie bialgebra,
     which is called a {\rm quasi-triangular Lie bialgebra} associated with $r$.
\item If $r$ is a skew-symmetric solution of the CYBE in $(\g, [-,-])$ then $(\g, [-,-],
     \tilde{\Delta}_{r})$ is a Lie bialgebra, which is called a {\rm triangular Lie
     bialgebra} associated with $r$.
\end{enumerate}
\end{pro}

Let $(\g, [-,-])$ be a Lie algebra. For any $r\in \g\otimes\g$,
we can also define linear map $r^{\sharp}: \g^{\ast}\rightarrow\g$ by
$\langle r^{\sharp}(\xi_{1}),\; \xi_{2}\rangle=\langle\xi_{1}\otimes\xi_{2},\; r\rangle$,
linear map $r^{\natural}: \g^{\ast}\rightarrow\g$ by
$\langle\xi_{1},\; r^{\natural}(\xi_{2})\rangle=-\langle\xi_{1}\otimes\xi_{2},\; r\rangle$,
for any $\xi_{1}, \xi_{2}\in\g^{\ast}$, and linear map $\mathcal{I}=r^{\sharp}-r^{\natural}:
\g^{\ast}\rightarrow\g$.

\begin{defi}\label{def:fact-liebia}
Let $(\g, [-,-])$ be a Lie algebra, $r\in\g\otimes\g$, and $(\g, [-,-], \tilde{\Delta}_{r})$
be a quasi-triangular Lie bialgebra associated with $r$. If $\mathcal{I}=
r^{\sharp}-r^{\natural}: \g^{\ast}\rightarrow\g$ is an isomorphism of vector spaces,
then $(\g, [-,-], \tilde{\Delta}_{r})$ is called a {\rm factorizable Lie bialgebra}.
\end{defi}

Next, we consider the Lie bialgebras constructed from Novikov bialgebras.
Recall that a {\it right Novikov algebra} $(B, \circ)$ is a vector space $B$
with a binary operation $\circ: B\rightarrow B\otimes B$, such that
\begin{align*}
&\qquad\qquad\qquad b_{1}\circ(b_{2}\circ b_{3})=b_{2}\circ(b_{1}\circ b_{3}),\\
&(b_{1}\circ b_{2})\circ b_{3}-b_{1}\circ(b_{2}\circ b_{3})
=(b_{1}\circ b_{3})\circ b_{2}-b_{1}\circ(b_{3}\circ b_{2}),
\end{align*}
for any $b_{1}, b_{2}, b_{3}\in B$. The operads of (left) Novikov algebras and right
Novikov algebras are the operadic Koszul dual each other. There is a Lie algebra structure
on the tensor product of a Novikov algebra and a right Novikov algebra.

\begin{pro}[\cite{HBG}]\label{pro:lie-nov}
Let $(A, \diamond)$ be a Novikov algebra and $(B, \circ)$ be a right Novikov algebra.
Define a bracket $[-,-]$ on the vector space $A\otimes B$ by
$$
[a_{1}\otimes b_{1},\; a_{2}\otimes b_{2}]=(a_{1}\diamond a_{2})\otimes(b_{1}\circ b_{2})
-(a_{2}\diamond a_{1})\otimes(b_{2}\circ b_{1}),
$$
for any $a_{1}, a_{2}\in A$ and $b_{1}, b_{2}\in B$. Then $(A\otimes B, [-,-])$
is a Lie algebra, which is called the {\it Lie algebra induced by $(A, \diamond)$
and $(B, \circ)$}.
\end{pro}

Recall that a {\it right Novikov coalgebra} is a pair $(B, \Delta)$ where
$B$ is a vector space and $\Delta: B\rightarrow B\otimes B$ is a linear map satisfying
\begin{align*}
&\qquad\qquad\qquad(\id\otimes\Delta)\Delta=(\tau\otimes\id)(\id\otimes\Delta)\Delta,\\
&(\Delta\otimes\id)\Delta-(\id\otimes\tau)(\Delta\otimes\id)\Delta
=(\id\otimes\Delta)\Delta-(\id\otimes\tau)(\id\otimes\Delta)\Delta.
\end{align*}

For the dual version of Proposition \ref{pro:lie-nov}, we have

\begin{pro}[\cite{HBG}]\label{pro:colie-conov}
Let $(A, \delta)$ be a Novikov coalgebra and $(B, \Delta)$ be a right Novikov coalgebra.
Define a linear map $\tilde{\Delta}: A\otimes B\rightarrow(A\otimes B)\otimes(A\otimes B)$ by
$$
\tilde{\Delta}(a\otimes b)=(\id\otimes\id-\tau)\Big(\sum_{(a)}\sum_{(b)}
(a_{(1)}\otimes b_{(1)})\otimes(a_{(2)}\otimes b_{(2)})\Big),
$$
for any $a\in A$ and $b\in B$, where $\delta(a)=\sum_{(a)}a_{(1)}\otimes a_{(2)}$
and $\Delta(b)=\sum_{(b)}b_{(1)}\otimes b_{(2)}$ in the Sweedler notation.
Then $(A\otimes B, \tilde{\Delta})$ is a Lie coalgebra.
\end{pro}

Let $(B, \circ)$ be a right Novikov algebra. A bilinear form $\omega(-,-)$ on
$(B, \circ)$ is called {\it invariant} if it satisfies
$$
\omega(b_{1}\circ b_{2},\; b_{3})+\omega(b_{1},\; b_{2}\circ b_{3}+b_{3}\circ b_{2})=0,
$$
for any $b_{1}, b_{2}, b_{3}\in B$. A {\it quadratic right Novikov algebra}, denoted by
$(B, \circ, \omega)$, is a right Novikov algebra $(B, \circ)$ together
with a symmetric invariant nondegenerate bilinear form $\omega(-,-)$.
Let $(B, \circ, \omega)$ be a quadratic right Novikov algebra. Then the bilinear form
$\omega(-,-)$ can naturally expand to the tensor product $B\otimes B\otimes\cdots\otimes B$,
i.e.,
$$
\omega(-,-):\qquad (\underbrace{B\otimes\cdots\otimes B}_{\mbox{\tiny $k$-fold}})\otimes
(\underbrace{B\otimes\cdots\otimes B}_{\mbox{\tiny $k$-fold}})\longrightarrow\Bbbk,
$$
$\omega(b_{1}\otimes b_{2}\otimes\cdots\otimes b_{k},\ \ b'_{1}\otimes b'_{2}
\otimes\cdots\otimes b'_{k})=\sum_{i=1}^{k}\sum_{j=1}^{k}\omega(b_{i}, b_{j})$,
for any $b_{1}, b_{2},\cdots, b_{k}, b'_{1}, b'_{2},\cdots, b'_{k}\in B$.
Then $\omega(-,-)$ on $B\otimes B\otimes\cdots\otimes B$ is also a symmetric
nondegenerate bilinear form. Thus, for any quadratic right Novikov algebra
$(B, \circ, \omega)$, we can define a linear map $\Delta_{\omega}: B\rightarrow B\otimes B$
by $\omega(\Delta_{\omega}(b_{1}),\; b_{2}\otimes b_{3})=\omega(b_{1},\; b_{2}\circ b_{3})$,
for any $b_{1}, b_{2}, b_{3}\in B$. One can check that $(B, \Delta_{\omega})$
is a right Novikov coalgebra.

\begin{ex}\label{ex:rigNov}
Let $(B, \circ)$ be the $2$-dimensional right Novikov algebra with a basis $\{x_{1},
x_{2}\}$ whose multiplication is given by
$$
x_{1}\circ x_{1}=0,\quad x_{1}\circ x_{2}=-2x_{1},\quad
x_{2}\circ x_{1}=x_{1},\quad x_{2}\circ x_{2}=x_{2}.
$$
Define a bilinear form $\omega(-,-)$ on $(B, \circ)$ by
$\omega(x_{1}, x_{1})=\omega(x_{2}, x_{2})=0$, $\omega(x_{1}, x_{2})=\omega(x_{2},
x_{1})=1$. This bilinear form is symmetric nondegenerate and invariant, and so that,
$(B, \circ, \omega)$ is a quadratic right Novikov algebra. Then the bilinear form
$\omega(-,-)$ induces a right Novikov coalgebra structure $\Delta_{\omega}: B\rightarrow
B\otimes B$: $\Delta_{\omega}(x_{1})=x_{1}\otimes x_{1}$, $\Delta_{\omega}(x_{2})
=x_{1}\otimes x_{2}-2x_{2}\otimes x_{1}$.
\end{ex}

\begin{pro}[\cite{HBG}]\label{pro:liebia-novbia}
Let $(A, \diamond, \delta)$ be a Novikov bialgebra and $(B, \circ, \omega)$ be a
quadratic right Novikov algebra, $(A\otimes B, [-,-])$ be the induced Lie algebra
by $(A, \diamond)$ and $(B, \circ)$. Define a linear map $\tilde{\Delta}:
A\otimes B\rightarrow(A\otimes B)\otimes(A\otimes B)$ by
$$
\tilde{\Delta}(a\otimes b)=(\id\otimes\id-\tau)\Big(\sum_{(a)}\sum_{(b)}
(a_{(1)}\otimes b_{(1)})\otimes(a_{(2)}\otimes b_{(2)})\Big),
$$
for any $a\in A$ and $b\in B$, where $\delta(a)=\sum_{(a)}a_{(1)}\otimes a_{(2)}$
and $\Delta_{\omega}(b)=\sum_{(b)}b_{(1)}\otimes b_{(2)}$ in the Sweedler notation.
Then $(A\otimes B, [-,-], \tilde{\Delta})$ is a Lie bialgebra, which is called the
{\rm Lie bialgebra induced by $(A, \diamond, \delta)$ and $(B, \circ, \omega)$}.
\end{pro}

Next, we consider the relation between the solutions of the NYBE in a Novikov algebra
and the solutions of the CYBE in the induced Lie algebra.
Let $(B, \circ, \omega)$ be a quadratic right Novikov algebra and
$\{e_{1}, e_{2},\cdots, e_{n}\}$ is a basis of $B$. Since $\omega(-,-)$ is
symmetric nondegenerate, we get a basis $\{f_{1}, f_{2},\cdots, f_{n}\}$ of $B$,
which is called the dual basis of $\{e_{1}, e_{2},\cdots, e_{n}\}$ with respect to
$\omega(-,-)$, by $\omega(e_{i}, f_{j})=\delta_{ij}$, where $\delta_{ij}$ is the
Kronecker delta. Here, for some special solutions of the NYBE in a Novikov algebra, we have

\begin{pro}\label{pro:NYBE-CYBE}
Let $(A, \diamond)$ be a Novikov algebra and $(B, \circ, \omega)$ be a quadratic
right Novikov algebra, and $(A\otimes B, [-,-])$ be the induced Lie algebra.
Suppose that $r=\sum_{i} x_{i}\otimes y_{i}\in A\otimes A$ is a solution
of the NYBE in $(A, \diamond)$, $\mathfrak{s}(r)$ is invariant, $\{e_{1}, e_{2},\cdots,
e_{n}\}$ is a basis of $B$ and $\{f_{1}, f_{2},\cdots, f_{n}\}$ is the dual basis of
$\{e_{1}, e_{2},\cdots, e_{n}\}$ with respect to $\omega(-,-)$. Then
$$
\widehat{r}=\sum_{i, j}(x_{i}\otimes e_{j})\otimes(y_{i}\otimes f_{j})
\in(A\otimes B)\otimes(A\otimes B)
$$
is a solution of the CYBE in $(A\otimes B, [-,-])$, and $\widehat{r}+\tau(\widehat{r})$
is $\ad_{A\otimes B}$-invariant.
\end{pro}

\begin{proof}
First, we have
\begin{align*}
&\;[\widehat{r}_{12},\; \widehat{r}_{13}]+[\widehat{r}_{12},\; \widehat{r}_{23}]
+[\widehat{r}_{13},\; \widehat{r}_{23}]\\
=&\;\sum_{i,j}\sum_{p,q}\Big((x_{i}\diamond x_{j}\otimes e_{p}\diamond e_{q}
-x_{i}\diamond x_{j}\otimes e_{q}\diamond e_{p})\otimes(y_{i}\otimes f_{p})
\otimes(y_{j}\otimes f_{q})\\[-4mm]
&\qquad\qquad+(x_{i}\otimes e_{p})\otimes(y_{i}\diamond x_{j}\otimes f_{p}\circ e_{q}
-x_{j}\diamond y_{i}\otimes e_{q}\circ f_{s})\otimes(y_{j}\otimes f_{q})\\[-1mm]
&\qquad\qquad+(x_{i}\otimes e_{p})\otimes(x_{j}\otimes e_{q})\otimes
(y_{j}\diamond y_{j}\otimes f_{p}\circ f_{q}-y_{j}\diamond y_{i}\otimes f_{q}\circ f_{p})\Big).
\end{align*}
Moreover, for given $s, u, v\in\{1, 2,\cdots, n\}$, we have
\begin{align*}
\omega\Big(e_{s}\otimes e_{u}\otimes e_{v},\;
\sum_{p,q}e_{p}\circ e_{q}\otimes f_{p}\otimes f_{q}\Big)
&=\omega(e_{s},\; e_{u}\circ e_{v}),\\[-2mm]
\omega\Big(e_{s}\otimes e_{u}\otimes e_{v},\;
\sum_{p,q} e_{q}\circ e_{p}\otimes f_{p}\otimes f_{q}\Big)
&=\omega(e_{s},\; e_{v}\circ e_{u}),\\[-2mm]
\omega\Big(e_{s}\otimes e_{u}\otimes e_{v},\;
\sum_{p,q} e_{p}\otimes f_{p}\circ e_{q}\otimes f_{q}\Big)
&=-\omega(e_{s},\; e_{u}\diamond e_{v}+e_{v}\circ e_{u}).
\end{align*}
By the nondegeneracy of $\omega(-,-)$, we get
$$
\sum_{p,q}e_{p}\otimes f_{p}\circ e_{q}\otimes f_{q}=-\sum_{p,q}e_{p}\circ
e_{q}\otimes f_{p}\otimes f_{q}+e_{q}\circ e_{p}\otimes f_{p}\otimes f_{q}.
$$
Similarly, we also have
\begin{align*}
&\qquad\qquad\sum_{p,q}e_{p}\otimes e_{q}\otimes f_{p}\circ f_{q}
=-\sum_{p,q}(e_{p}\otimes e_{q}\circ f_{p}\otimes f_{q}
+e_{p}\otimes e_{q}\otimes f_{q}\circ f_{p}),\\[-2mm]
&\sum_{p,q}e_{p}\otimes e_{q}\circ f_{p}\otimes f_{q}
=\sum_{p,q}e_{q}\circ e_{p}\otimes f_{p}\otimes f_{q},\qquad
\sum_{p,q}e_{p}\otimes e_{q}\otimes f_{q}\circ f_{p}
=\sum_{p,q}e_{p}\circ e_{q}\otimes f_{p}\otimes f_{q}.
\end{align*}
Thus, we have
\begin{align*}
&\;[\widehat{r}_{12},\; \widehat{r}_{13}]+[\widehat{r}_{12},\; \widehat{r}_{23}]
+[\widehat{r}_{13},\; \widehat{r}_{23}]\\
=&\;\sum_{p,q}\sum_{i,j}\Big((x_{i}\diamond x_{j}\otimes e_{p}\circ e_{q})\otimes
(y_{i}\otimes f_{p})\otimes(y_{j}\otimes f_{q})
-(x_{i}\otimes e_{p}\circ e_{q})\otimes
(y_{i}\circ x_{j} \otimes f_{p})\otimes (y_{j}\otimes f_{q})\\[-4mm]
&\qquad\qquad-(x_{i}\otimes e_{p}\circ e_{q})\otimes(x_{j}\otimes f_{p})
\otimes(y_{i}\diamond y_{j}\otimes f_{q})
-(x_{i}\otimes e_{p}\circ e_{q})\otimes(x_{j}\otimes f_{p})\otimes
(y_{j}\diamond y_{i}\otimes f_{q})\Big)\\
&\;-\sum_{p, q}\sum_{i,j}\Big((x_{j}\diamond x_{i}\otimes e_{q}\circ e_{p})\otimes
(y_{i}\otimes f_{p})\otimes(y_{j}\otimes f_{q})
+(x_{i}\otimes e_{q}\circ e_{p})\otimes(y_{i}\diamond x_{j}\otimes f_{p})
\otimes(y_{j}\otimes f_{q})\\[-4mm]
&\qquad\qquad+(x_{i}\otimes e_{q}\circ e_{p})\otimes(x_{j}\diamond y_{i}\otimes f_{p})
\otimes (y_{j}\otimes f_{q})
+(x_{i}\otimes e_{q}\circ e_{p})\otimes(x_{j}\otimes f_{p})\otimes
(y_{i}\diamond y_{j}\otimes f_{q})\Big)
\end{align*}
Since $\mathfrak{s}(r)$ is invariant, i.e., $-\sum_{j}\big((y_{i}\diamond x_{j})\otimes y_{j}
+x_{j}\otimes(y_{i}\diamond y_{j})+x_{j}\otimes(y_{j}\diamond y_{i})\big)
=\sum_{j}\big((y_{i}\diamond y_{j})\otimes x_{j}
+y_{j}\otimes(y_{i}\diamond x_{j})+y_{j}\otimes(x_{j}\diamond y_{i})\big)$, we get
\begin{align*}
&\;\sum_{i,j}\Big((x_{i}\diamond x_{j})\otimes y_{i}\otimes y_{j}
-x_{i}\otimes(y_{i}\diamond x_{j})\otimes y_{j}
-x_{i}\otimes x_{j}\otimes(y_{i}\diamond y_{j})
-x_{i}\otimes x_{j}\otimes(y_{j}\diamond y_{i})\Big)\\[-2mm]
=&\;\sum_{i,j}\Big((x_{i}\diamond x_{j})\otimes y_{i}\otimes y_{j}
+x_{i}\otimes(y_{i}\diamond y_{j})\otimes x_{j}
+x_{i}\otimes y_{j}\otimes(y_{i}\diamond x_{j})
+x_{i}\otimes y_{j}\otimes(x_{j}\diamond y_{i})\Big)\\[-2mm]
=&\;(\id\otimes\tau)(\mathbf{N}_{r})
\end{align*}
Thus, $[\widehat{r}_{12},\; \widehat{r}_{13}]+[\widehat{r}_{12},\; \widehat{r}_{23}]
+[\widehat{r}_{13},\; \widehat{r}_{23}]=0$, i.e., $\widehat{r}$ is a solution of
the CYBE in $(A\otimes B, [-,-])$ if $r$ is a solution of the NYBE in $(A, \diamond)$.

Next, we show that $\widehat{r}+\tau(\widehat{r})$ is $\ad_{A\otimes B}$-invariant.
For any $a\in A$ and $e_{p}\in B$, $p=1, 2,\cdots, n$, we have
\begin{align*}
&\big(\id\otimes\ad_{A\otimes B}(a\otimes e_{p})
+\ad_{A\otimes B}(a\otimes e_{p})\otimes\id\big)(\widehat{r}+\tau(\widehat{r}))\\
=&\;\sum_{i, j}\Big((x_{i}\otimes e_{j})\otimes[a\otimes e_{p},\; y_{i}\otimes f_{j}]
+[a\otimes e_{p},\; x_{i}\otimes e_{j}]\otimes(y_{i}\otimes f_{j})\\[-5mm]
&\qquad\quad+(y_{i}\otimes f_{j})\otimes[a\otimes e_{p},\; x_{i}\otimes e_{j}]
+[a\otimes e_{p},\; y_{i}\otimes f_{j}]\otimes(x_{i}\otimes e_{j})\Big)\\
=&\;\sum_{i,j}\Big((x_{i}\otimes e_{j})\otimes((a\diamond y_{i})\otimes(e_{p}\circ f_{j}))
-(x_{i}\otimes e_{j})\otimes((y_{i}\diamond a)\otimes(f_{j}\circ e_{p}))\\[-4mm]
&\qquad\quad+((a\diamond x_{i})\otimes(e_{p}\circ e_{j}))\otimes(y_{i}\otimes f_{j})
-((x_{i}\diamond a)\otimes(e_{j}\circ e_{p}))\otimes(y_{i}\otimes f_{j})\\
&\qquad\quad+(y_{i}\otimes f_{j})\otimes((a\diamond x_{i})\otimes(e_{p}\circ e_{j}))
-(y_{i}\otimes f_{j})\otimes((x_{i}\diamond a)\otimes(e_{j}\circ e_{p}))\\[-1mm]
&\qquad\quad+((a\diamond y_{i})\otimes(e_{p}\circ f_{j}))\otimes(x_{i}\otimes e_{j})
-((y_{i}\diamond a)\otimes(f_{j}\circ e_{p}))\otimes(x_{i}\otimes e_{j})\Big).
\end{align*}
For given $s, t\in\{1, 2,\cdots, n\}$, we have
$$
\omega\Big(e_{s}\otimes e_{t},\; \sum_{j}(e_{p}\circ e_{j})\otimes f_{j}\Big)
=\omega(e_{s},\; e_{p}\circ e_{t})=\omega(e_{t},\; e_{p}\circ e_{s})
=\omega\Big(e_{s}\otimes e_{t},\; \sum_{j}f_{j}\otimes(e_{p}\circ e_{j})\Big).
$$
Thus, $\sum_{j}(e_{p}\circ e_{j})\otimes f_{j}=\sum_{j}f_{j}\otimes(e_{p}\circ e_{j})$
since $\omega(-,-)$ is nondegenerate. Similarly, we also have
$\sum_{j}(e_{p}\circ e_{j})\otimes f_{j}=\sum_{j}e_{j}\otimes(e_{p}\circ f_{j})
=\sum_{j}(e_{p}\circ f_{j})\otimes e_{j}$.
Moreover, since $\omega(-,-)$ is invariant, we have
\begin{align*}
&\omega\Big(e_{s}\otimes e_{t},\; \sum_{j}(e_{j}\circ e_{p})\otimes f_{j}\Big)
=\omega(e_{s},\; e_{t}\circ e_{p})=-\omega(e_{t},\; e_{p}\circ e_{s})
-\omega(e_{t},\; e_{s}\circ e_{p})\\[-4mm]
&\qquad\quad=-\omega\Big(e_{s}\otimes e_{t},\; \sum_{j}f_{j}\otimes(e_{p}\circ e_{j})\Big)
-\omega\Big(e_{s}\otimes e_{t},\; \sum_{j}f_{j}\otimes(e_{j}\circ e_{p})\Big)
\end{align*}
That is, $\sum_{j}f_{j}\otimes(e_{p}\circ e_{j})=-\sum_{j}(e_{j}\circ e_{p})\otimes f_{j}
-\sum_{j}f_{j}\otimes(e_{j}\circ e_{p})$. Similarly, we have
$\sum_{j}e_{j}\otimes(e_{p}\circ f_{j})=-\sum_{j}(e_{j}\circ e_{p})\otimes f_{j}
-\sum_{j}e_{j}\otimes(f_{j}\circ e_{p})$, $\sum_{j}f_{j}\otimes(e_{p}\circ e_{j})
=-\sum_{j}(f_{j}\circ e_{p})\otimes e_{j}-\sum_{j}f_{j}\otimes(e_{j}\circ e_{p})$
and $\sum_{j}e_{j}\otimes(e_{p}\circ f_{j})=-\sum_{j}(f_{j}\circ e_{p})\otimes e_{j}
-\sum_{j}e_{j}\otimes(f_{j}\circ e_{p})$.
Thus, $\sum_{j}(e_{j}\circ e_{p})\otimes f_{j}=\sum_{j}(f_{j}\circ e_{p})\otimes e_{j}$
and $\sum_{j}f_{j}\otimes(e_{j}\circ e_{p})=\sum_{j}e_{j}\otimes(f_{j}\circ e_{p})$.
Therefore, we have
\begin{align*}
&\;\big(\id\otimes\ad_{A\otimes B}(a\otimes e_{p})
+\ad_{A\otimes B}(a\otimes e_{p})\otimes\id\big)(\widehat{r}+\tau(\widehat{r}))\\
=&\;\sum_{i,j}\Big(((a\diamond x_{i})\otimes e_{j})\otimes(y_{i}\otimes(f_{j}\circ e_{p}))
+(x_{i}\otimes e_{j})\otimes((a\diamond y_{i})\otimes(f_{j}\circ e_{p}))\\[-4mm]
&\qquad\quad+(x_{i}\otimes e_{j})\otimes((y_{i}\diamond a)\otimes(f_{j}\circ e_{p}))
+((a\diamond y_{i})\otimes e_{j})\otimes(x_{i}\otimes(f_{j}\circ e_{p}))\\[-1mm]
&\qquad\quad+(y_{i}\otimes e_{j})\otimes((a\diamond x_{i})\otimes(f_{j}\circ e_{p}))
+(y_{i}\otimes e_{j})\otimes((x_{i}\diamond a)\otimes(f_{j}\circ e_{p}))\Big)\\[-1mm]
&+\sum_{i,j}\Big(((a\diamond x_{i})\otimes(e_{j}\circ e_{p}))\otimes(y_{i}\otimes f_{j})
+(x_{i}\otimes(e_{j}\circ e_{p}))\otimes((a\diamond y_{i})\otimes f_{j})\\[-4mm]
&\qquad\quad+(x_{i}\otimes(e_{j}\circ e_{p}))\otimes((y_{i}\diamond a)\otimes f_{j})
+((a\diamond y_{i})\otimes(e_{j}\circ e_{p}))\otimes(x_{i}\otimes f_{j})\\[-1mm]
&\qquad\quad+(y_{i}\otimes(e_{j}\circ e_{p}))\otimes((a\diamond x_{i})\otimes f_{j})
+(y_{i}\otimes(e_{j}\circ e_{p}))\otimes((x_{i}\diamond a)\otimes f_{j})\Big).
\end{align*}
Since $\mathfrak{s}(r)$ is invariant, i.e., $(a\diamond x_{i})\otimes y_{i}
+x_{i}\otimes(a\diamond y_{i})+x_{i}\otimes(y_{i}\diamond a)+(a\diamond y_{i})\otimes x_{i}
+y_{i}\otimes(a\diamond x_{i})+y_{i}\otimes(x_{i}\diamond a)=0$, we get
$\big(\id\otimes\ad_{A\otimes B}(a\otimes e_{p})
+\ad_{A\otimes B}(a\otimes e_{p})\otimes\id\big)(\widehat{r}+\tau(\widehat{r}))=0$,
$\widehat{r}+\tau(\widehat{r})$ is $\ad_{A\otimes B}$-invariant. The proof is finished.
\end{proof}

By the nondegeneracy and symmetry of bilinear form $\omega(-,-)$, one can check that
the element $\widehat{r}\in(A\otimes B)\otimes(A\otimes B)$ defined in Proposition
\ref{pro:NYBE-CYBE} is skew-symmetric if $r$ is skew-symmetric. Thus, we have
the following corollary directly.

\begin{cor}[\cite{HBG}]\label{cor:sNYBE-sCYBE}
Let $(A, \diamond)$ be a Novikov algebra and $(B, \circ, \omega)$ be a quadratic
right Novikov algebra, $(A\otimes B, [-,-])$ be the induced Lie algebra.
If $r=\sum_{i} x_{i}\otimes y_{i}\in A\otimes A$ is a skew-symmetric solution of the
NYBE in $(A, \diamond)$, then
$$
\widehat{r}=\sum_{i, j}(x_{i}\otimes e_{j})\otimes(y_{i}\otimes f_{j})
\in(A\otimes B)\otimes(A\otimes B).
$$
is a skew-symmetric solution of the CYBE in $(A\otimes B, [-,-])$.
\end{cor}

We give the main conclusion of this section.

\begin{thm}\label{thm:indu-sLiebia}
Let $(A, \diamond, \delta)$ be a Novikov bialgebra and $(B, \circ, \omega)$ be a quadratic
right Novikov algebra, and $(A\otimes B, [-,-], \tilde{\Delta})$ be the induced
Lie bialgebra by $(A, \diamond, \delta)$ and $(B, \circ, \omega)$. Then, we have
\begin{enumerate}\itemsep=0pt
\item $(A\otimes B, [-,-], \tilde{\Delta})$ is quasi-triangular if
     $(A, \diamond, \delta)$ is quasi-triangular;
\item $(A\otimes B, [-,-], \tilde{\Delta})$ is triangular if
     $(A, \diamond, \delta)$ is triangular;
\item $(A\otimes B, [-,-], \tilde{\Delta})$ is factorizable if
     $(A, \diamond, \delta)$ is factorizable.
\end{enumerate}
\end{thm}

\begin{proof}
First, $(1)$ and $(2)$ follow from Proposition \ref{pro:NYBE-CYBE} and
Corollary \ref{cor:sNYBE-sCYBE}. Second, if $(A, \diamond, \delta)$ is factorizable,
i.e., there exists an element $r\in A\otimes A$ such that $\delta=\delta_{r}$,
$\mathfrak{s}(r)$ is invariant, and the map $\mathcal{I}=r^{\sharp}-r^{\natural}:
A^{\ast}\rightarrow A$ is an isomorphism of vector spaces, where $\delta_{r}$ is
given by Eq. \eqref{cobnov}. Suppose $\{e_{1}, e_{2},\cdots, e_{n}\}$ is a basis
of $B$ and $\{f_{1}, f_{2},\cdots, f_{n}\}$ is the dual basis of $\{e_{1}, e_{2},
\cdots, e_{n}\}$ with respect to $\omega(-,-)$. We need show that $\widehat{\mathcal{I}}
=\widehat{r}^{\sharp}-\widehat{r}^{\natural}: (A\otimes B)^{\ast}\rightarrow A\otimes B$ is
an isomorphism of vector spaces. Denote $\kappa:=\sum_{j}e_{j}\otimes f_{j}\in B\otimes B$.
Define $\kappa^{\sharp}: B^{\ast}\rightarrow B$ by $\langle\kappa^{\sharp}
(\eta_{1}),\; \eta_{2}\rangle=\langle\eta_{1}\otimes\eta_{2},\; \kappa\rangle$,
for any $\eta_{1}, \eta_{2}\in B$. Then, one can check that $\kappa^{\sharp}$ is
a linear isomorphism and $\langle\kappa^{\sharp}(\eta_{1}),\; \eta_{2}\rangle
=\langle\eta_{1}\otimes\eta_{2},\; \kappa\rangle
=\langle\eta_{2}\otimes\eta_{1},\; \kappa\rangle
=\langle\kappa^{\sharp}(\eta_{2}),\; \eta_{1}\rangle$.
Therefore, for any $\xi_{1}, \xi_{2}\in A$ and $\eta_{1}, \eta_{2}\in B$,
\begin{align*}
\langle\widehat{r}^{\natural}(\xi_{1}\otimes\eta_{1}),\; \xi_{2}\otimes\eta_{2}\rangle
&=-\sum_{i,j}\langle(\xi_{2}\otimes\eta_{2})\otimes(\xi_{1}\otimes\eta_{1}),\;
(x_{i}\otimes e_{j})\otimes(y_{i}\otimes f_{j})\rangle\\[-2mm]
&=\Big(-\sum_{i}\langle\xi_{2}, x_{i}\rangle\langle\xi_{1}, y_{i}\rangle\Big)
\Big(\sum_{j}\langle\eta_{1}, f_{j}\rangle\langle\eta_{2}, e_{j}\rangle\Big)\\[-2mm]
&=\langle r^{\natural}(\xi_{1}),\; \xi_{2}\rangle
\langle\kappa^{\sharp}(\eta_{2}),\; \eta_{1}\rangle\\
&=\langle r^{\natural}(\xi_{1})\otimes\kappa^{\sharp}(\eta_{1}),\;
\xi_{2}\otimes\eta_{2}\rangle.
\end{align*}
That is, $\widehat{r}^{\natural}=r^{\natural}\otimes\kappa^{\sharp}$, Similarly,
we have $\widehat{r}^{\sharp}=r^{\sharp}\otimes\kappa^{\sharp}$.
Thus, $\widehat{\mathcal{I}}=\widehat{r}^{\sharp}-\widehat{r}^{\natural}
=(r^{\sharp}-r^{\natural})\otimes\kappa^{\sharp}$ is an isomorphism of vector spaces.
The proof is finished.
\end{proof}

Let $(A, \diamond)$ be a Novikov algebra and $r\in A\otimes A$.
We the following diagram is commutative:
$$\xymatrix@C=3cm@R=0.6cm{
\txt{$r$ \\ {\tiny a solution of the NYBE}\\ {\tiny in
$(A, \diamond)$ of type ($\star$)}} \ar[d]_{{\rm Pro.}~\ref{pro:NYBE-CYBE}}
^{{\rm Cor.}~\ref{cor:sNYBE-sCYBE}} \ar[r]^{{\rm Pro.}~\ref{pro:tri-bialg},
\quad {\rm Pro.}~\ref{pro:inv-bia}} &
\txt{$(A, \diamond, \delta_{r})$ \\ {\tiny a Novikov bialgebra}}
\ar[d]^{{\rm Pro.}~\ref{pro:colie-conov}} \\
\txt{$\widehat{r}$ \\ {\tiny a solution of the CYBE}\\ {\tiny in $(A\otimes B, [-,-])$ of
type ($\star$)}} \ar[r]^{{\rm Pro.}~\ref{pro:lie-bia}\qquad\quad} &
\txt{$(A\otimes B, [-,-], \tilde{\Delta})=(A\otimes B, [-,-], \tilde{\Delta}_{\widehat{r}})$ \\
{\tiny the induced Lie bialgebra}}}
$$
where type ($\star$) means $(1)$ $\mathfrak{s}(r)$ is invariant, or $(2)$ $r$ is
skew-symmetric, or $(3)$ $\mathfrak{s}(r)$ is invariant and $\mathcal{I}=r^{\sharp}
-r^{\natural}$ is an isomorphism of vector spaces.

\begin{ex}\label{ex:ind-lietri}
Consider the 2-dimensional Novikov algebra $(A=\Bbbk\{e_{1}, e_{2}\}, \diamond)$, where
$-e_{1}\diamond e_{2}=e_{2}=e_{2}\diamond e_{1}$, $e_{1}\diamond e_{1}=
e_{2}\diamond e_{2}=0$. One can check that all solutions of the NYBE in this Novikov algebra
have the form $r=k(e_{1}\otimes e_{2}-e_{2}\otimes e_{1})+le_{2}\otimes e_{2}$ for
$k, l\in\Bbbk$. In this case, $\mathfrak{s}(r)$ is not invariant. We consider
a skew-symmetric solution $r=e_{1}\otimes e_{2}-e_{2}\otimes e_{1}$, and get a triangular
Novikov bialgebra $(A, \diamond, \delta_{r})$, where $\delta_{r}: A\rightarrow A\otimes A$
is given by $\delta_{r}(e_{1})=-e_{2}\otimes e_{1}$ and $\delta_{r}(e_{2})
=-e_{2}\otimes e_{2}$.

Let $(B, \circ, \omega)$ be the quadratic right Novikov algebra given in Example
\ref{ex:rigNov}, i.e., $B=\Bbbk\{x_{1}, x_{2}\}$, $x_{1}\circ x_{1}=0$,
$x_{1}\circ x_{2}=-2x_{1}$, $x_{2}\circ x_{1}=x_{1}$, $x_{2}\circ x_{2}=x_{2}$, and
$\omega(x_{1}, x_{1})=\omega(x_{2}, x_{2})=0$, $\omega(x_{1}, x_{2})=\omega(x_{2}, x_{1})=1$.
Then we get a right Novikov coalgebra $(B, \Delta_{\omega})$, where $\Delta_{\omega}(x_{1})
=x_{1}\otimes x_{1}$, $\Delta_{\omega}(x_{2})=x_{1}\otimes x_{2}-2x_{2}\otimes x_{1}$.
Moreover, by Proposition \ref{pro:liebia-novbia}, we get a Lie bialgebra $(A\otimes B,
[-,-], \tilde{\Delta})$, where
\begin{align*}
&\qquad [e_{1}\otimes x_{1},\; e_{2}\otimes x_{2}]=e_{2}\otimes x_{1}
=[e_{1}\otimes x_{2},\; e_{2}\otimes x_{1}],\\
&\qquad\qquad\qquad [e_{1}\otimes x_{2},\; e_{2}\otimes x_{2}]=-2e_{2}\otimes x_{2},\\
&\tilde{\Delta}(e_{1}\otimes x_{1})=(e_{1}\otimes x_{1})\otimes(e_{2}\otimes x_{1})
-(e_{2}\otimes x_{1})\otimes(e_{1}\otimes x_{1}),\\
&\tilde{\Delta}(e_{1}\otimes x_{2})=(e_{1}\otimes x_{2})\otimes(e_{2}\otimes x_{1})
-(e_{2}\otimes x_{1})\otimes(e_{1}\otimes x_{2})\\[-1mm]
&\qquad\qquad\qquad+2(e_{2}\otimes x_{2})\otimes(e_{1}\otimes x_{1})
-2(e_{1}\otimes x_{1})\otimes(e_{2}\otimes x_{2}),\\
&\tilde{\Delta}(e_{2}\otimes x_{2})=3(e_{2}\otimes x_{2})\otimes(e_{2}\otimes x_{1})
-3(e_{2}\otimes x_{1})\otimes(e_{2}\otimes x_{2}).
\end{align*}
On the other hand, in the quadratic right Novikov algebra $(B, \circ, \omega)$,
$\{y_{1}=x_{2},\; y_{2}=x_{1}\}$ is the dual basis of $\{x_{1}, x_{2}\}$ with
respect to $\omega(-,-)$. Let
$$
\widehat{r}=(e_{1}\otimes x_{1})\otimes(e_{2}\otimes x_{2})
-(e_{2}\otimes x_{1})\otimes(e_{1}\otimes x_{2})
+(e_{1}\otimes x_{2})\otimes(e_{2}\otimes x_{1})
-(e_{2}\otimes x_{2})\otimes(e_{1}\otimes x_{1}).
$$
Then, $\widehat{r}$ is a skew-symmetric solution of the CYBE in $(A\otimes B, [-,-])$,
and it induces a comultiplication $\tilde{\Delta}_{\widehat{r}}$ on $A\otimes B$ such that
$(A\otimes B, [-,-], \tilde{\Delta}_{\widehat{r}})$ is a triangular Lie bialgebra.
By direct calculation, we get $\tilde{\Delta}_{\widehat{r}}=\tilde{\Delta}$.
Thus, the diagram above is commutative.
\end{ex}

\begin{ex}\label{ex:ind-liefect}
Consider the 4-dimensional Novikov algebra $(A=\Bbbk\{e_{1}, e_{2}, e_{3}, e_{4}\},
\diamond)$, where $e_{1}\diamond e_{1}=e_{2}$, $e_{1}\diamond e_{4}=e_{2}-2e_{3}$,
$e_{4}\diamond e_{1}=-2e_{2}+e_{3}$, $e_{4}\diamond e_{4}=e_{3}$. Then, one can check
that $r=e_{1}\otimes e_{3}+e_{2}\otimes e_{4}$ is a solution of the NYBE in this
Novikov algebra, and $r+\tau(r)$ is invariant. Thus, by Example \ref{ex:fact-novbia},
we get a factorizable Novikov bialgebra $(A, \diamond, \delta_{r})$, where the nonzero
comultiplication $\delta_{r}: A\rightarrow A\otimes A$ is given by
$\delta_{r}(e_{1})=e_{2}\otimes e_{2}$, $\delta_{r}(e_{4})=-e_{3}\otimes e_{3}$.

Let $(B, \circ, \omega)$ be the quadratic right Novikov algebra given in Example
\ref{ex:rigNov}. By Proposition \ref{pro:liebia-novbia}, we get a Lie bialgebra
$(A\otimes B, [-,-], \tilde{\Delta})$, where
\begin{align*}
&\qquad [e_{1}\otimes x_{1},\; e_{1}\otimes x_{2}]=-3e_{2}\otimes x_{1}
=[e_{1}\otimes x_{2},\; e_{4}\otimes x_{1}],\\
&\qquad [e_{1}\otimes x_{1},\; e_{4}\otimes x_{2}]=3e_{3}\otimes x_{1}
=[e_{4}\otimes x_{2},\; e_{4}\otimes x_{1}],\\
&\qquad\qquad [e_{1}\otimes x_{2},\; e_{4}\otimes x_{2}]
=3e_{2}\otimes x_{2}-3e_{3}\otimes x_{2},\\
&\tilde{\Delta}(e_{1}\otimes x_{2})=3(e_{2}\otimes x_{1})\otimes(e_{2}\otimes x_{2})
-3(e_{2}\otimes x_{2})\otimes(e_{2}\otimes x_{1}),\\
&\tilde{\Delta}(e_{4}\otimes x_{2})=3(e_{3}\otimes x_{2})\otimes(e_{3}\otimes x_{1})
-3(e_{3}\otimes x_{1})\otimes(e_{3}\otimes x_{2}).
\end{align*}
On the other hand, in the quadratic right Novikov algebra $(B, \circ, \omega)$,
$\{y_{1}=x_{2},\; y_{2}=x_{1}\}$ is the dual basis of $\{x_{1}, x_{2}\}$ with
respect to $\omega(-,-)$. Let
$$
\widehat{r}=(e_{1}\otimes x_{1})\otimes(e_{3}\otimes x_{2})
+(e_{1}\otimes x_{2})\otimes(e_{3}\otimes x_{1})
+(e_{2}\otimes x_{1})\otimes(e_{4}\otimes x_{2})
+(e_{2}\otimes x_{2})\otimes(e_{4}\otimes x_{1}).
$$
Then, $\widehat{r}$ is a symmetric solution of the CYBE in $(A\otimes B, [-,-])$,
and one can check that it is $\ad_{A\otimes B}$-invariant. Thus, the solution $\widehat{r}$
induces a comultiplication $\tilde{\Delta}_{\widehat{r}}$ on $A\otimes B$ such that
$(A\otimes B, [-,-], \tilde{\Delta}_{\widehat{r}})$ is a quasi-triangular Lie bialgebra.
By direct calculation, we get $\tilde{\Delta}_{\widehat{r}}=\tilde{\Delta}$, and
$\widehat{\mathcal{I}}=\widehat{r}^{\sharp}-\widehat{r}^{\natural}: (A\otimes B)^{\ast}
\rightarrow A\otimes B$ is given by
\begin{align*}
&\widehat{\mathcal{I}}(f_{1}\otimes y_{1})=e_{3}\otimes x_{2},\quad
\widehat{\mathcal{I}}(f_{1}\otimes y_{2})=e_{3}\otimes x_{1},\quad
\widehat{\mathcal{I}}(f_{2}\otimes y_{1})=e_{4}\otimes x_{2},\quad
\widehat{\mathcal{I}}(f_{2}\otimes y_{2})=e_{4}\otimes x_{1},\\
&\widehat{\mathcal{I}}(f_{3}\otimes y_{1})=e_{1}\otimes x_{2},\quad
\widehat{\mathcal{I}}(f_{3}\otimes y_{2})=e_{1}\otimes x_{1},\quad
\widehat{\mathcal{I}}(f_{4}\otimes y_{1})=e_{2}\otimes x_{2},\quad
\widehat{\mathcal{I}}(f_{4}\otimes y_{2})=e_{2}\otimes x_{1},
\end{align*}
where $\{f_{1}, f_{2}, f_{3}, f_{4}\}\subseteq A^{\ast}$ is the dual basis of
$\{e_{1}, e_{2}, e_{3}, e_{4}\}$ and $\{y_{1}, y_{2}\}\subseteq B^{\ast}$ is the dual
basis of $\{x_{1}, x_{2}\}$. Thus, $(A\otimes B, [-,-], \tilde{\Delta})=
(A\otimes B, [-,-], \tilde{\Delta}_{\widehat{r}})$ as factorizable Lie bialgebras.
\end{ex}

\bigskip
\noindent
{\bf Acknowledgements. } This work was financially supported by National
Natural Science Foundation of China (No. 11801141).

\smallskip
\noindent
{\bf Declaration of interests.} The authors have no conflicts of interest to disclose.

\smallskip
\noindent
{\bf Data availability.} Data sharing is not applicable to this article as no new data were
created or analyzed in this study.


\begin{thebibliography}{abc}
\bibitem{Agu} M. Aguiar,
    On the associative analog of Lie bialgebras,
    J. Algebra {\bf 244} (2001), 492--532.

\bibitem{Bai1} C. Bai,
    Left-symmetric bialgebras and an analogue of the classical Yang-Baxter equation,
    Commun. Contemp. Math. {\bf 10} (2008), 221--260.

\bibitem{Bai} C. Bai,
    Double constructions of Frobenius algebras, Connes cocycles and their duality,
    J. Noncommut. Geom. \textbf{4} (2010), 475--530.

\bibitem{BLST} C. Bai, G. Liu, Y. Sheng, R. Tang,
    Quasi-triangular, factorizable Leibniz bialgebras and relative Rota-Baxter operators,
    Forum Mathematicum (accepted) 2024.

\bibitem{BLP} C. Bai, H. Li, Y. Pei,
    $\phi_{\epsilon}$-coordinated modules for vertex algebras,
    J. Algebra {\bf 246} (2015), 211--242.

\bibitem{BK} B. Bakalov, V. Kac,
    Field algebras,
    Int. Math. Res. Not. {\bf 2003} (2003), 123--159.

\bibitem{BN} A. Balinsky, S. Novikov,
    Poisson brackets of hydrodynamic type, Frobenius algebras and Lie algebras,
    Sov. Math. Dokl. {\bf 32} (1985), 228--231.

\bibitem{BCHM} K. Benali, T. Chtioui, A. Hajjaji, S. Mabrouk,
    Bialgebras, the Yang-Baxter equation and Manin triples for mock-Lie algebras,
    Acta et Commentationes Universitatis Tartuensis de Mathematica {\bf 27} (2023), 211--233.

\bibitem{BCZ} L.A. Bokut, Y. Chen, Z. Zhang,
    Gr\"{o}bner-Shirshov bases method for Gelfand-Dorfman-Novikov algebras,
    J. Algebra Appl. {\bf 16} (2017), 1750001, 22pp.


\bibitem{BD} Y. Bruned and V. Dotsenko,
    Novikov algebras and multi-indices in regularity structures,
    arXiv: 2311.09091.

\bibitem{BHZ} Y. Bruned, M. Hairer, L. Zambotti,
    Algebraic renormalization of regularity structures,
    Invent. Math. {\bf 215} (2019), 1039--1156.

\bibitem{CH} Z. Cui, B. Hou,
    Factorizable mock-Lie bialgebras and factorizable Rota-Baxter mock-Lie bialgebras,
    submitted.
	
\bibitem{Dri} V.G. Drinfeld,
    Hamiltonian structures of Lie groups, Lie bialgebras and the geometric meaning of
    the classical Yang-Baxter equations,
    Soviet Math. Dokl. {\bf 27} (1983), 68--71.
	
\bibitem{GD1}I. Gelfand, I. Dorfman,
    Hamiltonian operators and algebraic structures related to them,
    Funct. Anal. Appl. {\bf 13} (1979), 248--262.
	
\bibitem{GD2} I. Gelfand, I. Dorfman,
    Hamiltonian operators and infinite dimensional Lie algebras,
    Funct. Anal. Appl. {\bf 15} (1981), 173--187.

	
\bibitem{HBG} Y. Hong, C. Bai, L. Guo,
    Infinite-dimensional Lie bialgebras via affinization of Novikov bialgebras and
    Koszul duality,
    Comm. Math. Phys. {\bf 401} (2023), 2011--2049.
	
\bibitem{HBG1} Y. Hong, C. Bai, L. Guo,
    Deformation families of Novikov bialgebras via differential antisymmetric
    infinitesimal bialgebras, arXiv: 2402.16155.

\bibitem{HBG2} Y. Hong, C. Bai, L. Guo,
    A bialgebrta theory of Gel'fand-Dorfman algebras with applications to Lie conformal bialgebras, arXiv: 2401.13608.

\bibitem{K1}  V. Kac,
    Vertex algebras for beginners, 2nd Edition,
    Amer. Math. Soc., Providence, RI, 1998.

\bibitem{Ko} J. Koszul,
    Domaines born\'{e}s homog\'{e}nes et orbites de transformations affines,
    Bull. Soc. Math. France, {\bf 89} (1961), 515--533.

\bibitem{LS} H. Lang, Y. Sheng,
    Factorizable Lie bialgebras, quadratic Rota-Baxter Lie algebras and Rota-Baxter Lie
    bialgebras, Commun. Math. Phys. {\bf 397} (2023), 763--791.

\bibitem{LLB} Y. Lin, X. Liu, C. Bai,
    Differential antisymmetric infinitesimal bialgebras, coherent derivations and
    Poisson bialgebras,
    SIGMA Symmetry Integrability Geom. Methods Appl. {\bf 19} (2023), 018, 47pp.

\bibitem{LMW} M. Livernet, B. Mesablishvili, R. Wisbauer,
    Generalised bialgebras and entwined monads and comonads,
    J. Pure Appl. Algebra {\bf 219} (2015), 3263--3278.

\bibitem{RS} N. Reshetikhin, M.A. Semenov-Tian-Shansky,
    Quantum R-matrices and factorization problems,
    J. Geom. Phys. {\bf 5} (1988), 533--550.

\bibitem{SW} Y. Sheng, Y. Wang,
    Quasi-triangular and factorizable antisymmetric infinitesimal bialgebras,
    J. Algebra {\bf 628} (2023), 415--433.

\bibitem{Xu} X. Xu,
    Quadratic conformal superalgebras,
    J. Algebra, {\bf 231} (2000), 1--38.
	
\bibitem{Vin} E. Vinberg,
    Convex homogeneous cones,
    Trans. Mosc. Math. Soc. {\bf 12} (1963), 340--403.

\bibitem{WBLS} Y. Wang, C. Bai, J. Liu, Y, sheng,
    Quasi-triangular pre-Lie bialgebras, factorizable pre-Lie bialgebras and Rota-Baxter
    pre-Lie algebras,
    J. Geom. Phys. {\bf 199} (2024), 105146, 22pp.

\bibitem{Zhe} V.N. Zhelyabin,
    Jordan bialgebras and their relation to Lie bialgebras,
    Algebra Logic, {\bf 36} (1997), 1--15.


\end{thebibliography}
\end{document}